\documentclass[11pt,leqno]{article}
\usepackage{graphicx, amsfonts, amsthm, amsxtra, verbatim, multicol, hyperref}
\usepackage{makecell}
\usepackage[mathscr]{euscript}
\begin{document}
%-------------------------------------------------------------------------

\newtheorem{theo}{Theorem}
\newtheorem{lem}{Lemma}
\newtheorem{exam}{Example}
\newtheorem{coro}{Corollary}
\newtheorem{defi}{Definition}
\newtheorem{axio}{I}

\newtheorem{prob}{Problem}
\newtheorem{lemm}{Lemma}
\newtheorem{prop}{Proposition}
\newtheorem{rem}{Remark}
\newtheorem{conj}{Conjecture}

\def\jacob{{\rm jacob}}          %the Jacobian ideal
\newcommand\Pn[1]{\mathbb{P}^{#1}}   %Projective space of dimension #1
\def\Z{\mathbb{Z}}                 %Integer  numbers
\def\Q{\mathbb{Q}}                   %Rational  numbers
\def\C{\mathbb{C}}                   %Complex numbers
\def\R{\mathbb{R}}                   %real numbers
\def\N{\mathbb{N}}                   %natural numbers
\def\P{\mathbb{P}}                   %projective space
\def\k{{\mathsf k}}                  %Arbitrary field
\def\K{{\mathsf K}}                  %Arbitrary field
\def\A{\mathbb{A}}                   %affine space
\def\F{{\cal F}}                      %Foliation
\def\one{\Theta}                 %one
\def\kernel{{\rm Kernel}}                %kernel
\def\image{{\rm Image}}                %kernel

\def\sing{{\rm sing}}

\def\O{{\cal O}}                 %
%-------------------------------------------------------------------------

\begin{center}
{\LARGE\bf 
Algebraic curves  and foliations
%On the module of foliations with a fixed invariant algebraic leaf
%Algebraic solution of holomorphic foliations
}
%\footnote{ 
%The text is still under construction, last modified 18th of January
%2006\\
%Math. classification: \\
%Keywords:
%} \\

\vspace{.25in} {\large {\sc César Camacho\footnote{EMAp-FGV, Praia de Botafogo 190, 
Rio de Janeiro,  22250-900, RJ, Brazil, {\tt cesar.camacho@fgv.br},}, Hossein Movasati\footnote{IMPA, Estrada Dona Castorina 110, Rio de Janeiro,  22460-320,  RJ,  Brazil,
{\tt hossein@impa.br}, } 
}} \\
\vspace{.1in} {With an appendix by \large {\sc Claus Hertling\footnote{Universit\"at Mannheim, B6, 26, 68159,  Mannheim, Germany, {\tt hertling@math.uni-mannheim.de.}

} }} \\
%Email: {\tt hossein@impa.br}
\begin{abstract}
Consider a field $\k$ of characteristic $0$, not necessarily algebraically closed, and a fixed algebraic curve $f=0$ defined by a tame polynomial $f\in\k[x,y]$ with only quasi-homogeneous singularities. We prove that the space of holomorphic foliations in the plane $\A^2_\k$ having $f=0$ as a fixed invariant curve is generated as $\k[x,y]$-module by at most four elements, three of them are the trivial foliations $fdx,fdy$ and $df$. Our proof is algorithmic and constructs the fourth foliation explicitly. Using Serre's GAGA and Quillen-Suslin theorem, 
we show that for a suitable  field extension $\K$ of $\k$ such a module over $\K[x,y]$ is actually generated by two elements, and therefore, such curves are free divisors in the
sense of K. Saito. After performing Groebner basis for this module, we observe that in many well-known examples, $\K=\k$.%, such a field extension is not needed, that is, $\K=\k$.  %this module is actually generated by two foliations, 
\end{abstract}
\end{center}
%such as curves studied by Lins Neto and Mendes-Pereira, and graphs of polynomial functions,  
%-------------------------------------------------------------------------
%---------------------------------------------------------------------

%------------------------------------------------------------------
\section{Introduction}
The geometric analysis of one dimensional foliations in the complex projective plane $\P^2_\C$, seen from a global perspective, is a wide open field of research. In particular, the study of foliations that admit an algebraic curve as an integral. In \cite{CLS92} foliations whose limit set is an algebraic 
curve with hyperbolic holonomy are characterized as rational pull backs of linear foliations. 
Curves of high degree can be leaves of foliations of lower degree, and there is no general statement explaining this phenomenon. For example 
A. Lins Neto in \cite{LinsNeto2002} considers  a curve which is a union of nine lines supported by two foliations of projective 
degree $4$. Studying the pencil generated by these two foliations he answers questions posed by Poincar\'e and Painlev\'e. \\
In this paper we explore the algebraic aspects of foliations supporting a fixed algebraic invariant curve. Let $\k$ be a subfield of $\C$ and $\Omega_{\A_\k^2}^1$ the space of 1-forms on $\A_\k^2$. An element $\omega\in\Omega_{\A_\k^2}^1$ induces the foliation $\omega=0$ in the  affine variety $\A^2_\k$. Let $f\in \k[x,y]$  be a polynomial  and  $C: f(x,y)=0$ the induced curve in $\A_\k^2$. We assume that $f$ is a tame polynomial. 
\begin{defi}(\cite{mo}, \cite[\S 10.6]{ho13})\rm.
\label{10/11/2020}
A polynomial $f\in \k[x,y]$ is called {\it tame} if the following property 
is satisfied. In the homogeneous decomposition of $f$
$$
f=f_{d}+\cdots+f_2+f_1+f_0,\ 
$$
into degree $i$ homogeneous pieces $f_i$ in the weighted ring $\k[x,y], \ \deg(x)=\alpha_1,\deg(y)=\alpha_2$ for some $\alpha_1,\alpha_2\in\N$, 
the last homogeneous piece $g:=f_{d}$ has finite dimensional 
Milnor vector space $V_g:=\frac{\k[x,y]}{\jacob(g)}$. \\ 
For $\alpha_1=\alpha_2=1$ this is equivalent to say that  $g$ induces $d$ distinct points 
in $\P_{\bar \k}^1$, that is $g=\prod_{i=1}^{d}(x-a_iy)$, where the $a_i\in \bar\k$ are pairwise different. In geometric terms this means that 
the line at infinity $\P ^1_{\bar\k}:=\P ^2_{\bar\k}\backslash \A_{\bar\k} ^2$ intersects the curve induced by $f=0$ in $\P ^2_\k$ transversely.  
\end{defi}%This has to do with the behavior of $C$  at infinity. For instance, if $C$ intersects  the line at infinity transversally then $f$ is tame. 
%We will need that a singularity $p$ of $C$ satisfies the following property:  the linear map of 
%multiplication by $f$ in ${\cal O}_{\A^2_\k,p}/\langle f_x,f_y\rangle$ have Jordan blocks of the same size. It For instance, nodal singularities are good. A point $p$ of $C$ is called a nodal singularity if the complex curve $C(\C)$ around $p$ is a union of$s$ smooth local curves transversal to each other at $p$. 
Let 
\begin{equation}
\label{13july2020}
E_f:=\left\{\omega\in \Omega_{\A_\k^2}^1\Big| df\wedge \omega=f\alpha,\ \hbox{ for some } \alpha\in \Omega_{\A_\k^2}^2\right \}.
\end{equation}
This is the $\k[x,y]$-module of differential $1$-forms $\omega$ such that the foliation induced by $\omega=0$ in $\A^2_\k$  leaves $f=0$ invariant. 
By degree of $\omega=Pdx+Qdy$ we mean the affine degree $\deg(\omega):={\rm max} \{\deg(P),\deg(Q)\}$. 
In this article we prove that:
\begin{theo}
\label{main1}
 If all the singularities of $f=0$ are  quasi-homogeneous  then  
 there exists $\omega_f\in E_f$ %of degree $\leq 2d-2\alpha_1-2\alpha_2$ 
 such that
$fdx,fdy, df, \omega_f$ generate the $\k[x,y]$-module $E_f$. 
\end{theo}
The first three foliations $fdx,fdy,df$ are obviously in $E_f$ and we call them trivial foliations. Our proof of Theorem \ref{13july2020}
is algorithmic and it computes $\omega_f$. The proof of Theorem \ref{main1} only works for curves with quasi-homogeneous singularities. This follows from an argument due to C. Hertling, see Appendix \ref{24/11/2020}.
The curve $x^5+y^5-x^2y^2=0$ (due to A'Campo) has a non quasi-homogeneous singularity at the origin and it turns out that the conclusion of Theorem \ref{main1} is false in this case. For further details see Example \ref{5nov2020}.  

We have  implemented the computation of $\omega_f$  in a computer. This $1$-form is usually of high degree 
and it is not clear how to produce foliations of lowest possible degree in $E_f$. In order to investigate this problem,  we write $E_f$ in the standard Groebner 
basis and take its minimal resolution. It turns out that in many interesting examples, $E_f$ is actually generated by two foliations  $\omega_0:=P_0dx+Q_0dy$ and 
$\omega_\infty:=P_\infty dx+Q_\infty dy$. This includes the curves studied by Lins Neto,  Mendes-Pereira, and graphs of polynomial functions. For this we have prepared the Table  \ref{30june2020}, for more details regarding this table 
see \S \ref{lopes2020}.

For smooth curves it is easy to see that the foliations given by 
$fdx,fdy,df$ generate $E_f$, see Proposition \ref{27oct2020}.  For curves with only quasi-homogeneous singularities (including smooth curves) 
Quillen-Suslin theorem implies that $E_f\otimes_\k \K$ is actually generated by two foliations $\omega_0$ and $\omega_\infty$, for some  
field extension $\K$ of $\k$. 
Computing $\omega_0$ and $\omega_\infty$, and in particular the field extension  $\K$  seems to be a new problem 
not treated in the literature.
% arguments used in the proof of Theorem \ref{main1} are mostly valid in higher dimensions $\A^n_\k$. 
%However, from the point of view of holomorphic foliations, one is only interested in those $\omega\in E_f$ such that the integrability condition $\omega\wedge d\omega=0$ holds. This trivially holds for $n=2$, but for $n>2$ it is a non-linear identity in $\omega$, and hence it is not clear how to interpret results like  Theorem \ref{main1} in this case.  
The arguments used in the proof of Theorem \ref{main1} are essentially valid in higher dimensions $\A^n_\k$. However, in this paper we are interested in holomorphic foliations, that is, in those $\omega\in E_f$ such that the integrability condition $\omega\wedge d\omega = 0$ holds. This trivially holds for $n=2$, but for $n>2$ it is a nonlinear identity in $\omega$ and it is not clear how to interpret results like Theorem \ref{main1} in this case.

The article is organized in the following way. In \S\ref{5/11/2020-1} we introduce a basis of monomials for the Milnor
vector space of $f$. In \S\ref{5/11/2020-2} we recall quasi-homogeneous singularities and apply K. Saito's  theorem to these singularities in order 
to get a local freeness statement. We prove Theorem \ref{main1} in  \S\ref{5/11/2020-3} and observe that for a curve with A'Campo singularity 
Theorem \ref{main1} does not hold. The examples presented in Table \ref{30june2020} and even more, are discussed in 
\S\ref{lopes2020}.  In \S\ref{5/11/2020-4} we use Quillen-Suslin theorem and Serre's GAGA in order to prove that $E_f\otimes_\k \K$ is actually free for some 
field extension $\k\subset \K$. 
We discuss this field extension in the case of a circle. The computer codes of the present paper are written in {\sc Singular}, see \cite{GPS01}. 
%For this see Appendix \ref{computercodes}.  
For the computations in this paper we have written the procedures
{\tt SyzFol, MinFol, BadPrV} of {\tt foliation.lib} which are available on the second author's webpage.
\footnote{ {\tt http://w3.impa.br/$\sim$hossein/foliation-allversions/foliation.lib}  }

Our sincere thanks go to J. V. Pereira for many useful discussions related to the topic of the present paper. We thank also C. Hertling 
for his comments and corrections to the present article and for writing Appendix \ref{24/11/2020}.  
%Claus Hertling

\begin{table}
\label{30june2020}
\begin{center}
\begin{tabular}{|c|c|}

\hline
Name
\thead{
$f(x,y)=0$
\\
$\left(
\begin{array}{*{2}{c}}
P_0 & Q_0  \\
P_{\infty}&  Q_{\infty}
\end{array}
\right)
$
}
%&
%\thead{ \tiny First integral of $\F(\omega_0)$ \\ \tiny First integral of $\F(\omega_\infty)$}
&
Figure  
\\

\hline
\hline
Riccati 
\thead{
$f(x)$
\\
$\left(
\begin{array}{*{2}{c}}
1 & 0  \\
0&  f(x)
\end{array}
\right)
$
}
%&
%\thead{ $x$ \\ $y$}
&
\thead{
\includegraphics[width=0.1\textwidth]{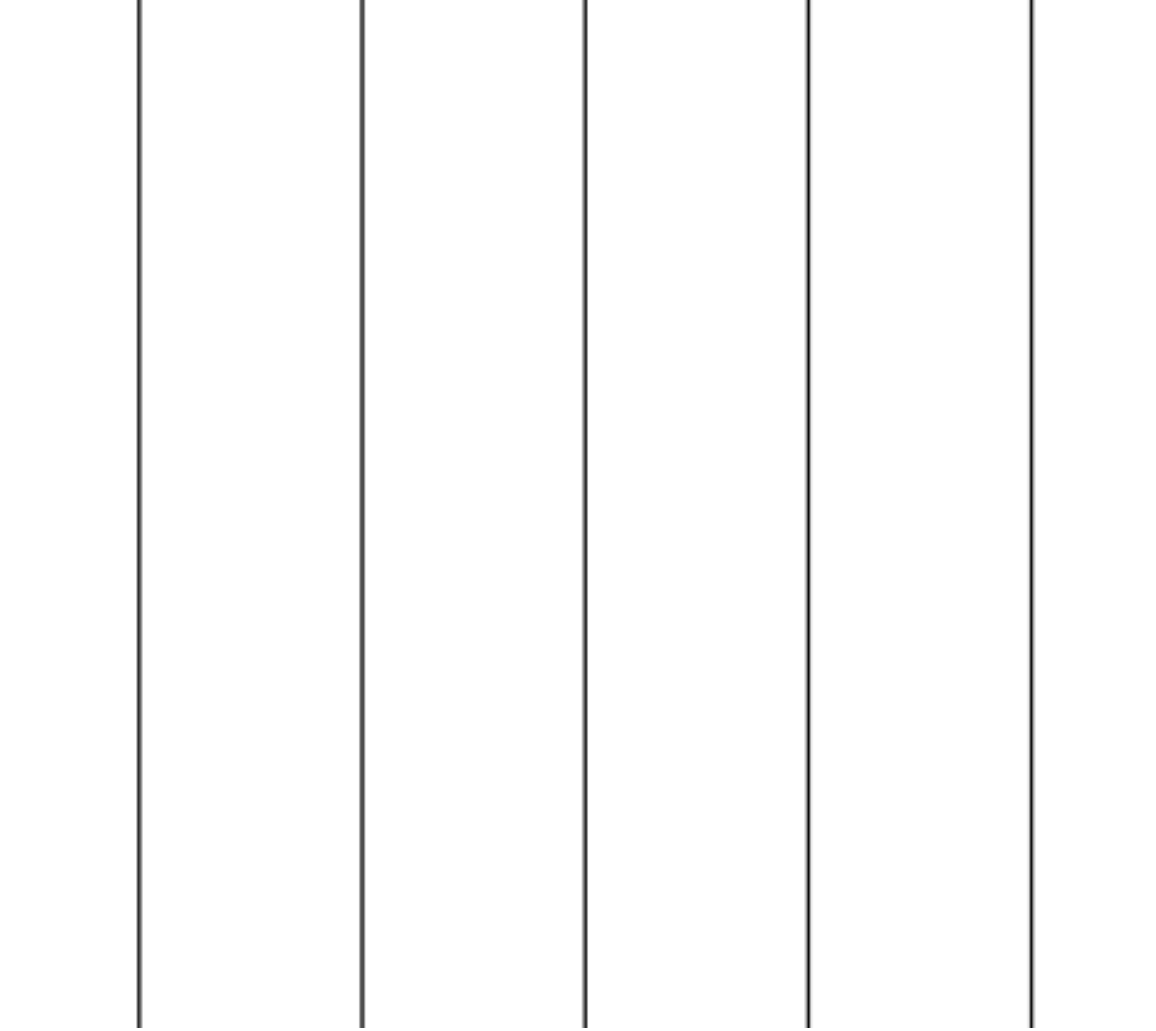}} 
\\

\hline
 Quasi-homogeneous
 \thead{
$f(x,y)$
\\
$\left(
\begin{array}{*{2}{c}}
f_x & f_y\\
-\frac{\alpha_2}{d} y &   \frac{\alpha_1}{d} x
\end{array}
\right)
$
}
%& 
%\thead{ $\frac{x^{\alpha_2}}{y^{\alpha_1}}$ \\  $f(x,y)$ }
&
\thead{
\includegraphics[width=0.1\textwidth]{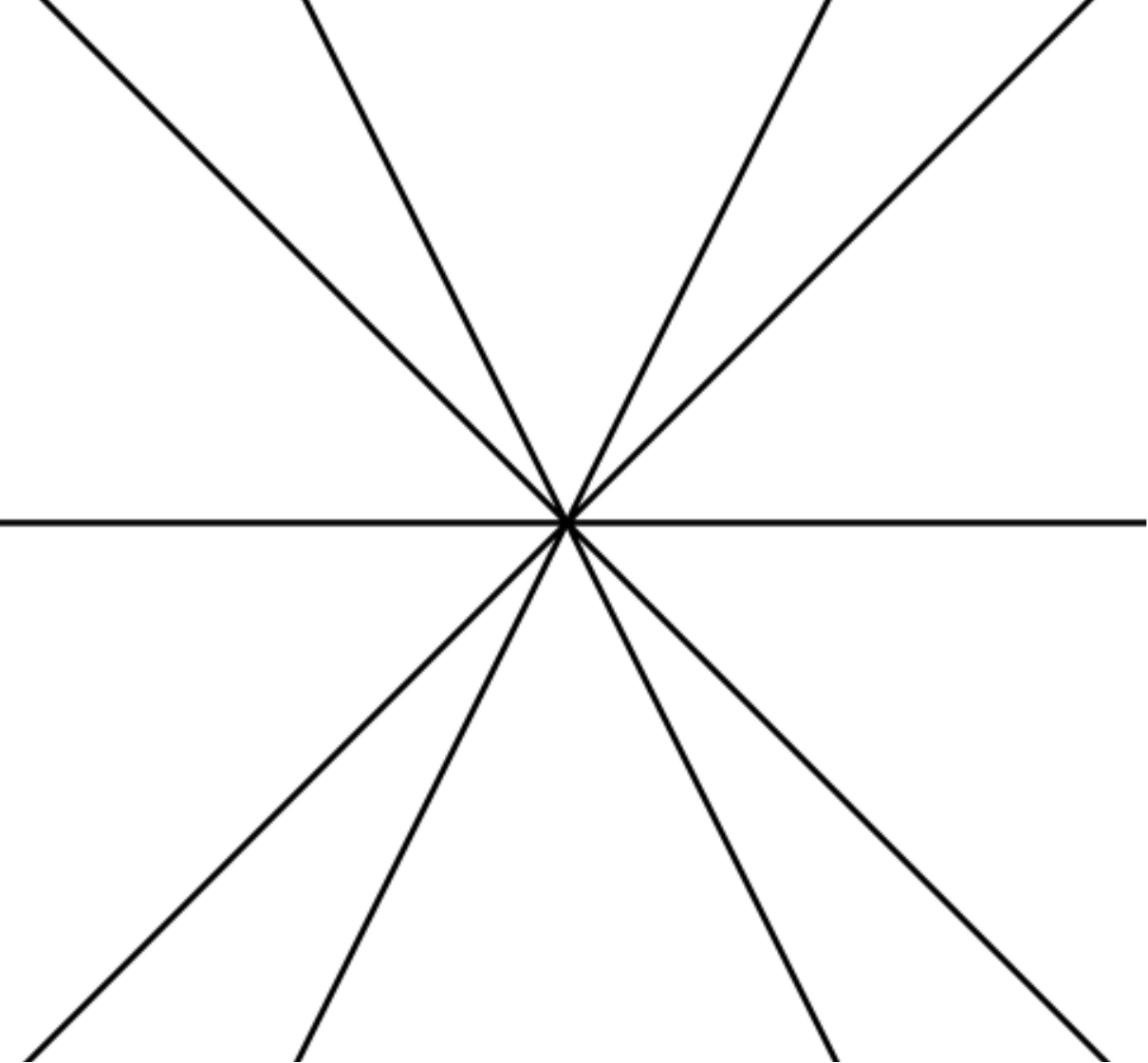}}

\\

\hline
Graph 
\thead{
$y-f(x)$
\\
$\left(
\begin{array}{*{2}{c}}
ya_1(x)+a_2(x) & a_{3}(x)  \\
yb_1(x)+b_2(x) & b_{3}(x)
\end{array}
\right)
$
}
%&
&
\thead{
\includegraphics[width=0.1\textwidth]{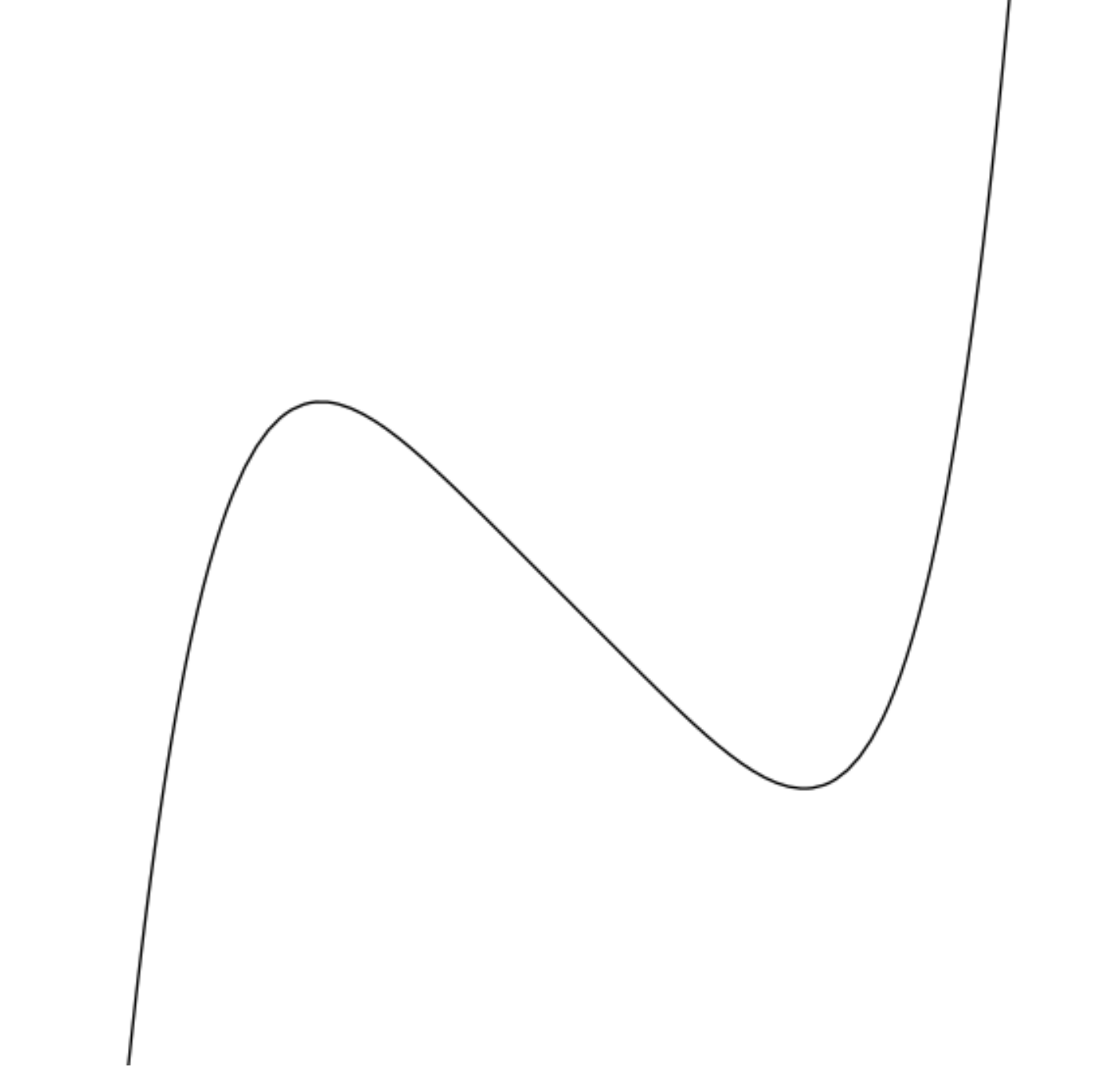} } 
\\

\hline

Lins Neto ($a=3$)
\thead{
$(x^a-1)(y^a-1)(x^a-y^a)$
\\
$\left(
\begin{array}{*{2}{c}}
 -y^{a+1}+y &  x^{a+1}-x \\
 -x^{a-1}(y^a-1)  & y^{a-1}(x^a-1)  
\end{array}
\right)
$
}
%&
%\thead{
%$\frac{x^{-a}(x^a-1)}{y^{-a} (y^a-1)}$
%\\
%$\frac{x^a-1}{y^a-1}$
%} 

&
\thead{
\includegraphics[width=0.1\textwidth]{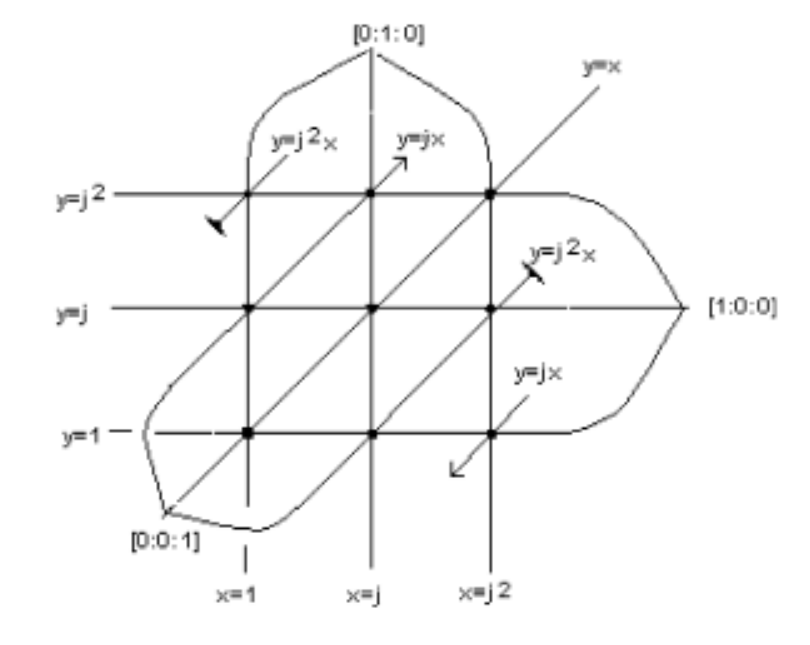}\\ $ (x^3-1)(y^3-1)(x^3-y^3)=0$ } 
\\ \hline

%Lins Neto
%\thead{
%$ 4y^2(1-3x)-4x^3+(3x^2+y^2)^2$
%\\
%$
%\left(
%\begin{array}{*{2}{c}}
%3\cdot x^{2}-3\cdot y^{2} & 4\cdot xy-2\cdot y \\
%-12\cdot xy+6\cdot y & 9\cdot x^{2}-y^{2}-4\cdot x
%\end{array}
%\right)
%$
%}
%& 
%&
%\thead{
%\includegraphics[width=0.1\textwidth]{figures/4_LinsNeto_Quartic.pdf}
%} 
%\\ \hline

Rose $k=\frac{1}{2}$
 \thead{
{\tiny $4x^4 + 8x^2y^2 + 4y^4- 4x^6 - 12x^4y^2 - 12x^2y^4 - 4y^6 - y^2$}
\\
{\tiny
$
\left(
\begin{array}{*{2}{c}}
8\cdot x^{3}-xy^{2} & 11\cdot x^{2}y+2\cdot y^{3}-y \\
5\cdot x^{2}y-4\cdot y^{3}+2\cdot y & x^{3}+10\cdot xy^{2}-x
\end{array}
\right)
$
}
}
%& 
%$\frac{(x^2 + y^2 - 1/3)^3}{36x^2 + 9y^2 - 4}$
&
\thead{
\includegraphics[width=0.1\textwidth]{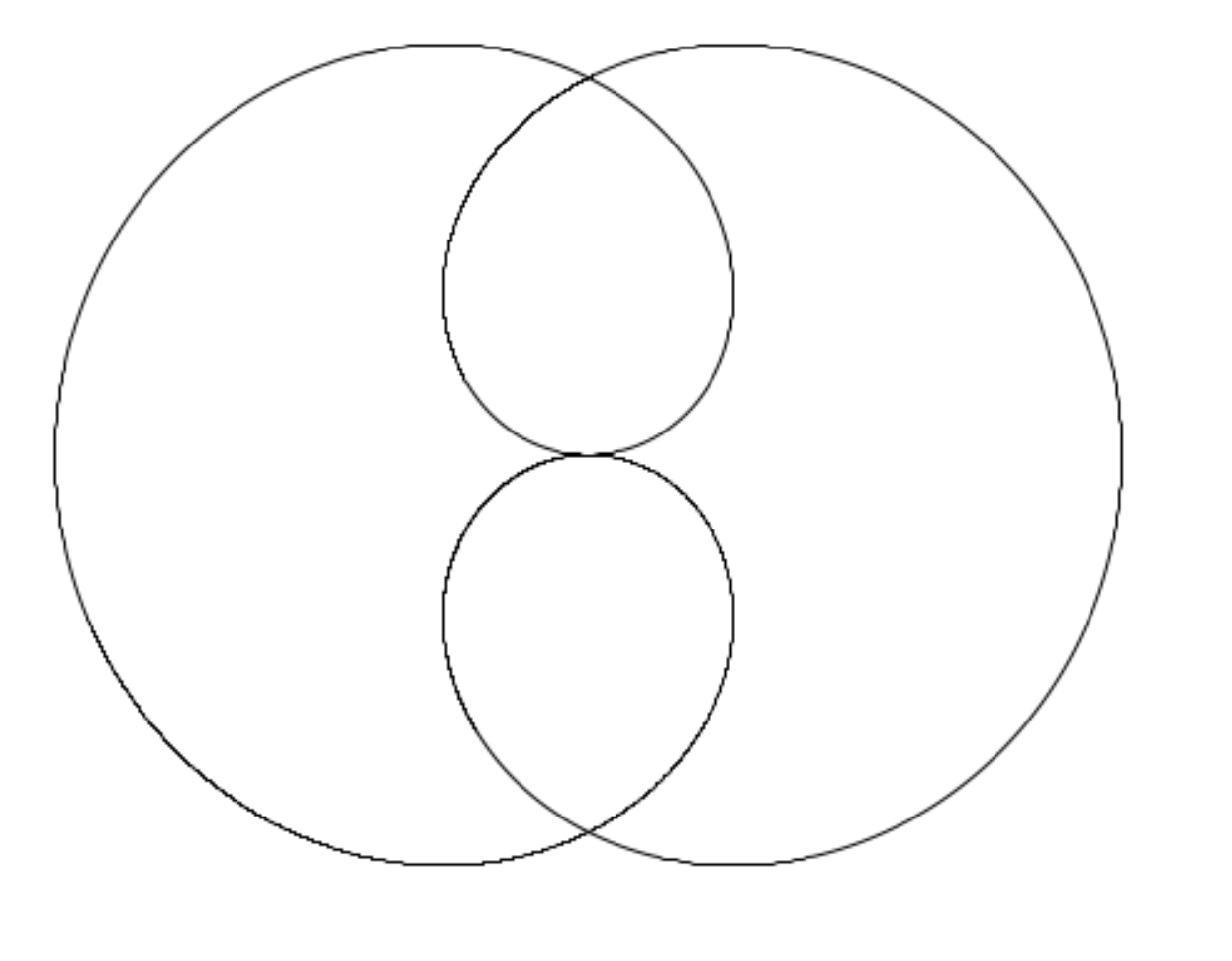}
} 
\\ \hline
\href{https://mathcurve.com/courbes2d/lissajous/lissajous.shtml}{Lissajous}
\thead{
%$8x4+2a4-8a2x2-4ay3-a3y+4a3y-a4$
$2(2x^2-a^2)^2-a(2y-a)^2(y+a)$ \\
$x=a \cos(3t), y=a \cos(4t)$
\\
{\tiny
$
\left(
\begin{array}{*{2}{c}}
-16\cdot xy^{2}-8a\cdot xy+8a^{2}\cdot x & 12\cdot x^{2}y+6a\cdot x^{2}-6a^{2}\cdot y-3a^{3} \\
-16\cdot x^{2}y-8a\cdot y^{2}+4a^{2}\cdot y+4a^{3} & 12\cdot x^{3}+6a\cdot xy-9a^{2}\cdot x
\end{array}
\right)
$
}
}
%& 
%$\frac{(y+a)(2y-a)^2}{(2x^2-a^2)^2}$ 
&
\thead{
\includegraphics[width=0.1\textwidth]{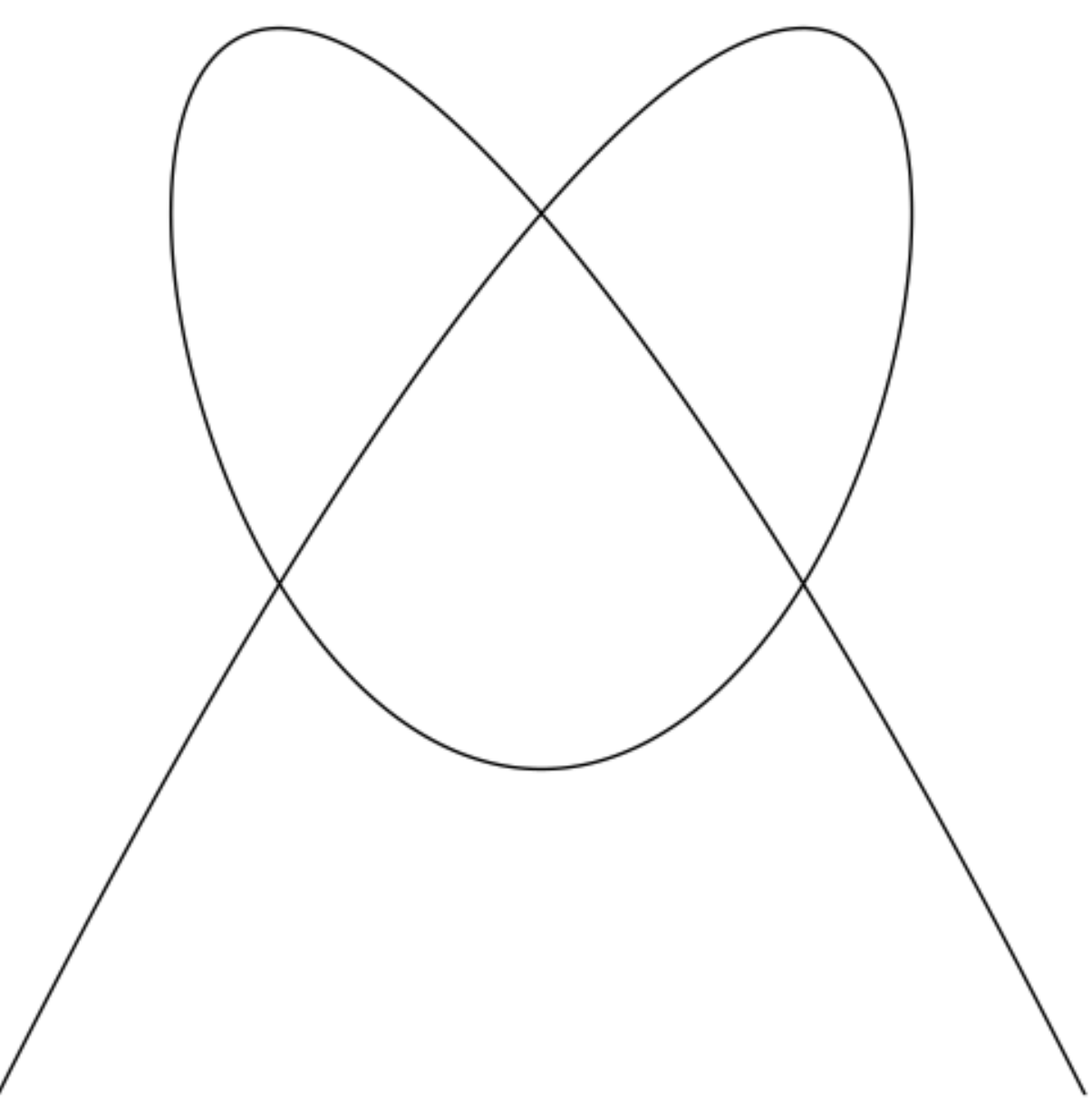}\\ $a=-1$
} 
\\ \hline

\href{https://mathcurve.com/courbes2d.gb/deltoid/deltoid.shtml}{Deltoid}
\thead{
$(x^2+y^2)^2+8ax(x^2-3y^2)+18a^2(x^2+y^2)-27a^4$
\\
{\tiny
$
\left(
\begin{array}{*{2}{c}}
x^{2}-3\cdot y^{2}+6a\cdot x+9a^{2} & 4\cdot xy-6a\cdot y \\
-4\cdot xy+6a\cdot y & 3\cdot x^{2}-y^{2}+6a\cdot x-9a^{2}
\end{array}
\right)
$
}
}
%& 
&
\thead{
\includegraphics[width=0.1\textwidth]{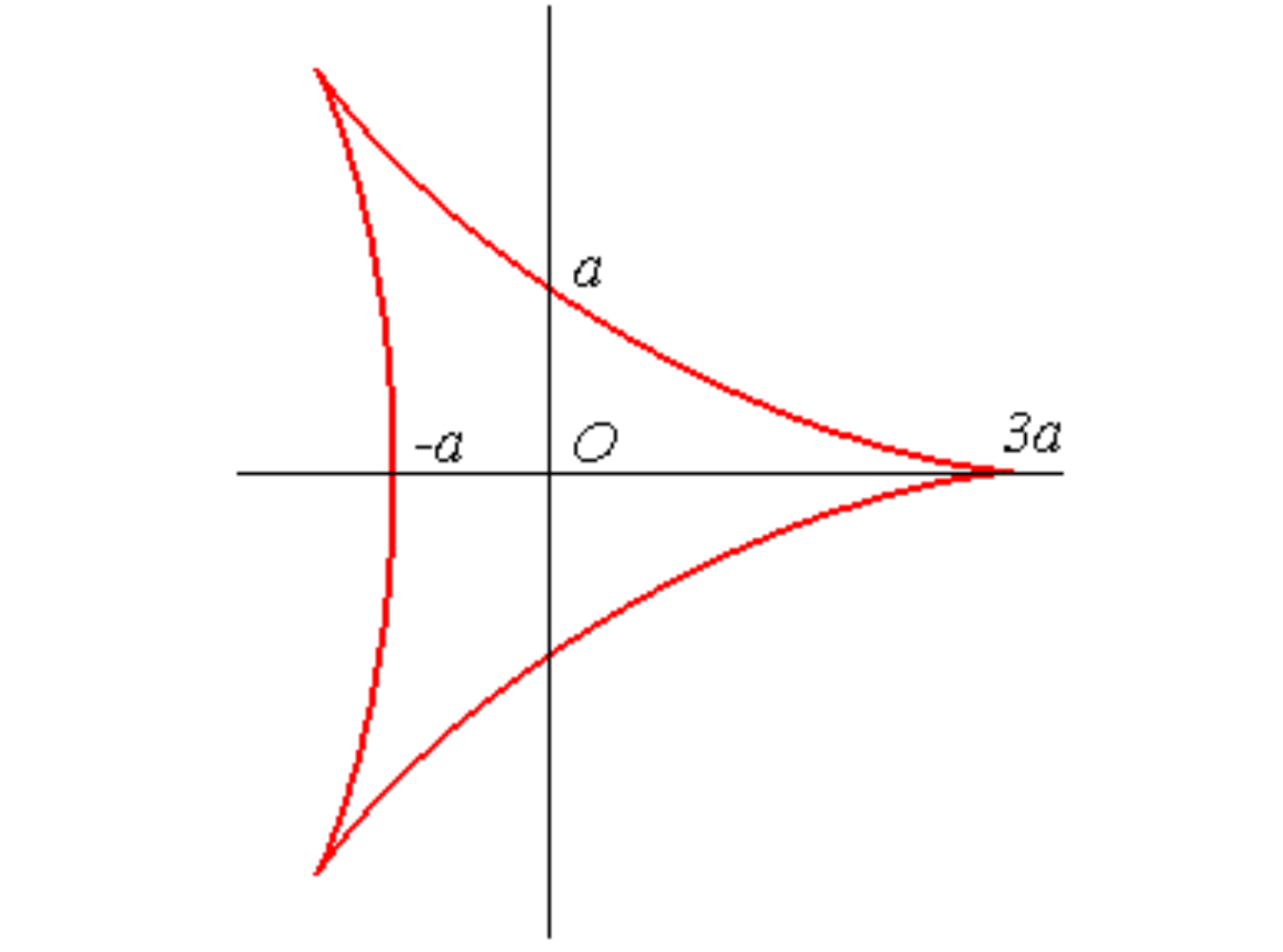}
} 
\\ \hline

\end{tabular}
\end{center}
\caption{Curves and foliations}
\end{table}

%--------------------------------------------------------------
\section{Milnor vector space}
\label{5/11/2020-1}
Recall the definition of a tame polynomial in Introduction. %By usual degree of a polynomial we mean the degree in the ring $\k[x,y],\ \deg(x)=\deg(y)=1$, and 
In the following by degree of a polynomial we mean the  weighted 
degree in $\k[x,y]$, $\deg(x)=\alpha_1,\ \deg(y)=\alpha_2$.
\begin{prop}
\label{alireza}
Let $f\in\k[x,y]$ be  tame polynomial of weighted degree $d$ in $\k[x,y]$, $\deg(x)=\alpha_1,\ \deg(y)=\alpha_2$. 
There is a basis of monomials $x^iy^j$ for the
Milnor vector space
$$
V_f:=\frac{\k[x,y]}{\jacob(f)}\cong \frac{\Omega^2_{\A_\k^2}}{df\wedge \Omega^1_{\A_\k^2}}
$$
with weighted degree $\leq 2d-2\alpha_1-2\alpha_2$, and among these monomials, only one  monomial is of the highest usual degree $2d-2\alpha_1-2\alpha_2$.
\end{prop}
\begin{proof}
The proposition for $f=g$ a homogeneous polynomial in a weighted ring $\k[x,y]$, $\deg(x)=\alpha_1,\ \deg(y)=\alpha_2$ is a classical fact due to 
V. Arnold and K. Saito, see \cite[Corollary 4, page 200]{arn1}. Actually in this case, all the 
monomials of usual degree $> 2d-2\alpha_1-2\alpha_2$ are zero in $V_g$.  
For an arbitrary $f$ with the last  homogeneous piece $g$ in the same weighted ring, we first find a monomial basis with the desired property for   
$V_g$ and it follows that the same set of monomials form a basis of $V_f$, see \cite[\S 6]{mo} or 
\cite[Proposition 10.7]{ho13}. In this context, monomials $x^iy^j$ with $i\alpha_1+j\alpha_2>2d-2\alpha_1-2\alpha_2$ are not necessarily zero in $V_f$ and we only know that they are equivalent in $V_f$ to polynomials of  degree $\leq 2d-2\alpha_1-2\alpha_2$.
 \end{proof}
%In this text we consider the degree of differential forms $\omega\in \Omega_{\A_\k}^i$ defined by
%$$
%\deg(dx)=\deg(dy)=1,\ \deg(dx\wedge dy)=2
%$$
%and the usual degree of polynomials. 
%-------------------------------------------------------------------------
\section{Quasi-homogeneous singularities}
\label{5/11/2020-2}
Let $f=0$ be a germ of curve singularity given by a germ of a holomorphic function  $f\in\O_{\C^2,0}$. 
\begin{theo}[K. Saito \cite{KSaito1980} page 270]
\label{2nov2020}
The $\O_{\C^2,p}$-module of holomorphic $1$-forms tangent to $f=0$ is freely generated if and only if it has two elements 
$\omega_0,\omega_\infty$ such that $\omega_0\wedge \omega_\infty=f dx\wedge dy$. 
\end{theo}
%It is worth noting that the above result is also true for algebraic curves $f=0$ in $\A^2_\k$ for an arbitrary field $\k\ubset \C$. 
%This menas that two $1$-forms $\omega_0, \omega_\infty\in E_f$ generate $E_f$
%as $\k[x,y]$-module if and only if $\omega_0\wedge \omega_\infty=fdx\wedge dy$. This for instance, follows, from a similar argument as in Proposition \ref{10nov2020}.  

Quasi-homogeneous singularities are the main example of singularities satisfying Theorem \ref{2nov2020}.  
\begin{defi}\rm
\label{10nov2020}
A germ of curve singularity $f=0$ is called quasi-homogeneous
if there is a holomorphic change of coordinates $(x,y)$ in $(\C^2,0)$ such that $(x(0),y(0))=(0,0)$ and 
$f(x,y)=a(x,y)\cdot g(x,y),\ \ a\in \O_{\C^2,0},a(0)\not=0$ and $g$ is a quasi-homogeneous polynomial in $x,y$, that is, there are $\alpha_1,\alpha_2\in \Z$ such that $g$ is homogeneous in the weighted ring 
$\C[x,y],\ \deg(x)=\alpha_1,\ \ \deg(y )=\alpha_2$.
\end{defi}
Let $\eta:=\frac{1}{\deg(g)}(\alpha_1 x dy-\alpha_2 ydx)$. We have $dg\wedge \eta=g dx\wedge dy$, and so by Theorem \ref{2nov2020}, 
the module of $1$-forms tangent to $g=0$ is freely generated by 
$dg, \eta$.
\begin{theo}[K. Saito \cite{Saito1971}]
 A germ of curve singularity $f=0$ is quasi-homogeneous if and only if $f$ belongs to the Jacobian ideal $\jacob(f):=\langle \frac{\partial f}{\partial x},\frac{\partial f}{\partial y} \rangle $.
\end{theo}

%----------------------------------------------------------
\section{Proof of Theorem \ref{main1} }
\label{5/11/2020-3}

In $V_f$ we consider the linear map $A_f$ defined by
$$
A_f:V_f\to V_f, A_f([P])=[Pf]
$$
Let $p(t)$ be the minimal polynomial of $A_f$. 
The critical values of $f$ are exactly the zeros of $p(t)$, see for instance \cite[Proposition 2.7 page 150]{Cox2005} \cite[\S 10.9]{ho13}. 
If $0\in\bar\k$ is not a critical value of $f$ then $\{f=0\}$ is a smooth curve.
Suppose that $f=0$ is singular, and hence, $A_f$ has a non-trivial kernel. We write $p(t)=tq(t)$.
By definition of a minimal polynomial
\begin{equation}
\label{10dec2020}
fq(f)=0, \ q(f)\not =0 \text{ in } V_f.
\end{equation}
The polynomial $q(f)$ is of big degree. Proposition \ref{alireza} implies that 
we can simplify $q(f)$ in $V_f$ and obtain
a polynomial $\Theta_f$ of weighted degree $\leq 2d-2\alpha_1-2\alpha_2$ which is equal to $q(f)$ in $V_f$. 
We obtain
\begin{equation}
f\Theta_fdx\wedge dy=df\wedge \omega_f,\ %\deg(\omega_f)\leq 2d-2\alpha_1-2\alpha_2+1,\ 
\end{equation}
%and call $\F(\omega_f)$ the universal foliation attached to $f$. 
Note that we do not have any control on the degree of $\omega_f$. 
Note also that $\Theta_f$ and $\omega_f$ depend on the monomial basis that we have chosen in Proposition \ref{alireza}.

%Let $p(t)$ be the minimal polynomial of multiplication by $f$ map $A_f$ in the Milnor vector space $V_f:=\k[x,y]/\jacob(f)$ of $f$.  
The exponent of the affine curve $f=0$ is the number $m$ in $p(t)=t^mQ(t), \ Q(0)\not =0$.
The exponent of a curve is also the maximum size of the Jordan blocks of $A_f$ associated to the eigenvalue $0$. The exponent of a critical point $p$ of $f$ with $f(p)=0$  is the minimum number $\tilde m$ such that $f^{\tilde m}$ is zero 
in the local Milnor vector space $V_p:={\cal O}_{\A_\k^2, p}/\jacob(f)$.

\begin{prop}\label{kimurasan}
Let $\k$ be an algebraically closed field of characteristic zero. 
We have an isomorphism
$$
V_f\cong \oplus_{p\in P_f } \frac{{\cal O}_{\A_\k^2,p}}{\jacob(f)},
$$
induced by canonical restriction, where $p$ runs through the set $P_f$ of critical points of $f$. 
In particular, each piece in the above summand is invariant under the multiplication by $f$ map $A_f$. 
Let $\lambda_1,\lambda_2,\ldots,\lambda_k$ be the critical values of $f$ and
$m_i,\ i=1,2,\ldots,k$ be  the maximum of exponents of the
singularities in $f(x,y)-\lambda_i=0$. Then
the minimal polynomial of $A_f$ is  
$p(t):=(t-\lambda_1)^{m_1}(t-\lambda_2)^{m_2}\cdots (t-\lambda_k)^{m_k}$.
\end{prop}
\begin{proof}
This is an immediate corollary of a well-known fact in commutative algebra, see \cite[Theorem 2.2 ]{Cox2005}. It also follows from   
Max Noether's theorem, see \cite[page 703]{gri}.  For the surjectivity of the restriction map,  we must modify the proof of Noether's theorem.    
Note that  $A_f-f(p) {\rm Id}$ restricted to  $\frac{{\cal O}_{\A_\k^2,p}}{\jacob(f)}$  is nilpotent of order which is the exponent  of $p$.
\end{proof}
Theorem \ref{main1} is a consequence of the following: 
\begin{theo}
\label{5feb07}
We have
\begin{enumerate}
\item 
If there exists $\one \in \kernel(A_f)$ such that $\one\cdot V_f=\ker(A_f)$ then $E_f$ is generated by $fdx,fdy,df,\omega_\one$,
%every holomorphic foliation which leaves $C$ invariant is of the form
%\begin{equation}
%\label{form}
%\F(P_1df+f\omega+P_2\omega_{\one})
%\end{equation}
where $f\one dx\wedge dy=df\wedge \omega_\one$. 
\item
If the Jordan blocks of $A_f$ associated to the eigenvalue $0$ have the same size then
$\one_f$ satisfies $\one_f\cdot V_f=\ker(A_f)$.
\item
If all the singularities of $C$ are quasi-homogeneous then
 the Jordan blocks of $A_f$ associated to the eigenvalue $0$ have the size $1$.
\end{enumerate}
\end{theo}
\begin{proof}
 Proof of 1. 
For $\omega\in E_f$ we have $f\Theta_1dx\wedge dy=df\wedge \omega$ for some polynomial $\Theta_1\in\k[x,y]$. This implies that 
$\Theta_1\in \ker(A_f)$ and by our hypothesis, we have $\Theta_1=P\one$ in $V_f$ for some $P\in\k[x,y]$. We write this as
$\Theta_1dx\wedge dy-P\one dx\wedge dy=df\wedge \eta$. After multiplication of this equality by $f$ we get
$(\omega-P\omega_{f}-f\eta)\wedge df=0$ which implies the result. De Rham lemma, see for instance \cite{sa76} in the case of homogeneous polynomials and \cite[Proposition 10.3]{ho13} for tame 
 polynomials, implies that $\omega$ is of the desired format.

Proof of 2. It is enough to prove  
that the sequence 
\begin{equation}
\label{13dec2020}
V_f\stackrel{q(A_f)}{\rightarrow} V_f\stackrel{A_f}{\rightarrow} V_f
\end{equation}
is exact. For this we can assume that $\k$ is algebraically closed, as the exactness of a sequence of vector spaces is independent of whether $\k$ is algebraically closed or not.   It follows from \eqref{10dec2020}  that  $\image(q(A_f))\subset \kernel(A_f)$. The non-trivial assertion is $\kernel(A_f)\subset \image(q(A_f))$.
Let $V_f=V_1\oplus V_2$ such that $A_f\mid_{V_1}$ is nilpotent and 
$A_f\mid_{V_2}$ is invertible. For this we simply write $A_f$ in Jordan block format, $V_1$ is constructed from the blocks with eigenvalue $0$ and $V_2$ from the remaining blocks.  From $p(t)=tq(t)=t^{m}Q(t),\ Q(0)\not=0$, it follows that $q(A_f)\mid_{V_1}=A_f^{m-1}\circ B_f$, where  $B_f:=Q(A_f)\mid_{V_1}: V_1\to V_1$ is invertible.  Therefore, we can assume that 
$p(t)=t^m, q(t)=t^{m-1}$ and $Q=1$.   Now, we use the following statement in linear algebra. Let $A$ be an $n\times n$ matrix with entries in $\bar\k$ and all eigenvalues equal to zero. Let also $m$ be the maximal size of its Jordan blocks. Then  $\kernel(A)=\image(A^{m-1})$ if and only if all the Jordan blocks of $A$ are of the same size. In order to see this fact, let 
$$
e_1^1,e_2^1,\ldots,e_{k_1}^1, e_{1}^2,e_{2}^2,\ldots,
e_{k_2}^2,\ldots\ldots, %e_{k_{m-1}}^{s-1}, 
e_{1}^s,e_{2}^s,\ldots, e_{k_s}^s, %\tilde e_i, i\in \Lambda
k_1\leq k_2\leq \cdots\leq k_s=m.
$$ 
be a basis of $\bar\k^n$ such that $A$ in each block $e_{1}^j,\ldots,e_{k_j}^j$
is a shifting map and $A_f(e_{k_j}^j)=0$. By definition of $m$ we have $\image(A^{m-1})\subset \kernel(A)$. Moreover,  $\kernel(A)\subset \image(A^{m-1})$ if and only if 
$k_1=k_2=\cdots=k_s$.

Proof of 3. 
The third part follows from Proposition \ref{kimurasan} and the fact that
the exponent of a quasi-homogeneous  singularity is $1$.
\end{proof}

\begin{rem}\rm
 Let $[\Theta_i],\ i=1,2,\ldots,a$ be a basis for the $\k$-vector space
$\ker(A_f)$ and write $df\wedge \omega_i=f\Theta_idx\wedge dy$. 
The $\k[x,y]$-module $E_f$ is generated by
$fdx,fdy,df,\omega_i,\ i=1,2,\ldots,a$.
 This is an immediate consequence of definitions.  
\end{rem}

\begin{rem}\rm
 A careful analysis of the proof of Theorem \ref{main1} shows that this theorem is true for curves $f=0$ such that the multiplication by $f$ map in  its Milnor vector space $V_f$ has Jordan blocks of the same size. Jordan blocks of size $1$ correspond to quasi-homogeneous singularities. In Appendix \ref{24/11/2020}, C. Hertling proves that there is no curve singularity such that multiplication by $f$ in its Milnor vector space has only Jordan blocks of size $2$.  
\end{rem}

\begin{prop}
\label{27oct2020}
A foliation which leaves a smooth curve $C\in \P^2_\k$ invariant is of the form
$$
\F(Pdf+f\omega), \ \omega\in\Omega_{\A_k^2}^1, \  P\in \Omega_{\A_k^2}^0
$$
in an affine chart $\A_\k^2\subset  \P^2_\k$. 
In other words the $\k[x,y]$-module $E_f$ defined in \eqref{13july2020} is generated by $fdx, fdy, df$.
\end{prop}
\begin{proof}
 We take a line in  $\P^2_\k$ which intersects $C$ transversally at $\deg(C)$ points, and hence, in the affine chart which is its complement, the curve $C$ is given by the tame polynomial $f(x,y)=0$. If $\omega\wedge df=fPdx\wedge dy$ then $P$ is in the kernel of the map $A_f$. Since $f=0$ is smooth, we conclude that $P$ itself is zero in $V_f$, and hence, 
 $Pdx\wedge dy=df\wedge\alpha$ for some $\alpha\in\Omega_{\A^2_\k}^1$. Therefore,  $df\wedge (f\alpha-\omega)=0$. 
 De Rham lemma for tame polynomials, see  \cite[Proposition 10.3]{ho13}, implies that $\omega$ is of the desired format.  
\end{proof}

\begin{exam}\rm
\label{5nov2020}
 The singularity $(0,0)$ of the curve $f:=x^5+y^5-x^2y^2=0$ is called A'Campo singularity and it is not quasi-homogeneous. The polynomial $f$ has two critical values  $s=0,-\frac{16}{3125}$. Over the critical value $s=0$, we have only the singularity $(0,0)\in \A^2_\k$ of Milnor number $11$. In this case  
 $A_f$ has $10$ Jordan blocks, one is of size $2$ and $9$ of size $1$. Over the critical value $s=-\frac{16}{3125}$ we have $5$ singularities of Milnor number $1$. In this case, $A_f$ has  $5$ Jordan blocks of size $1$. The Macaulay code gives us the following three generators of $E_f$:
 
 {\tiny 
 \begin{verbatim}
  loadPackage "VectorFields"
   R=QQ[x,y];
   f=f=x^5+y^5-x^2*y^2;
   derlog(ideal (f));
-->image | -25x2y2+6xy    -25x3y+5y3+4x2 -5x4+3xy2 |
-->      | -25xy3+5x3+4y2 -25x2y2+6xy    -5x3y+2y3 |
 \end{verbatim}
 }
 The columns $[P,Q]^{\rm tr }$ of the output matrix are vector fields $P\frac{\partial}{\partial x}+Q\frac{\partial}{\partial y}$. 
 Such three vector fields generate the $\Q[x,y]$-module of vector fields tangent to $f=0$. The corresponding $1$-forms must be written as $Pdy-Qdx$. Let $\omega_1,\omega_2,\omega_3$ be such $1$-forms. The minimal polynomial of $A_f$ turns out to be $t^2(3125t+16)$ and 
 $$
 \omega_f=-\frac{4}{25}x^2y\omega_1.
 $$
 This implies that $\omega_1$ is not in the module generated by $df,fdx,fdy,\omega_f$.  
 The  real locus of $f=0$ is depicted in 
 Figure \ref{8oct2020}.
 \begin{figure}[t]
\begin{center}
\includegraphics[width=0.3\textwidth]{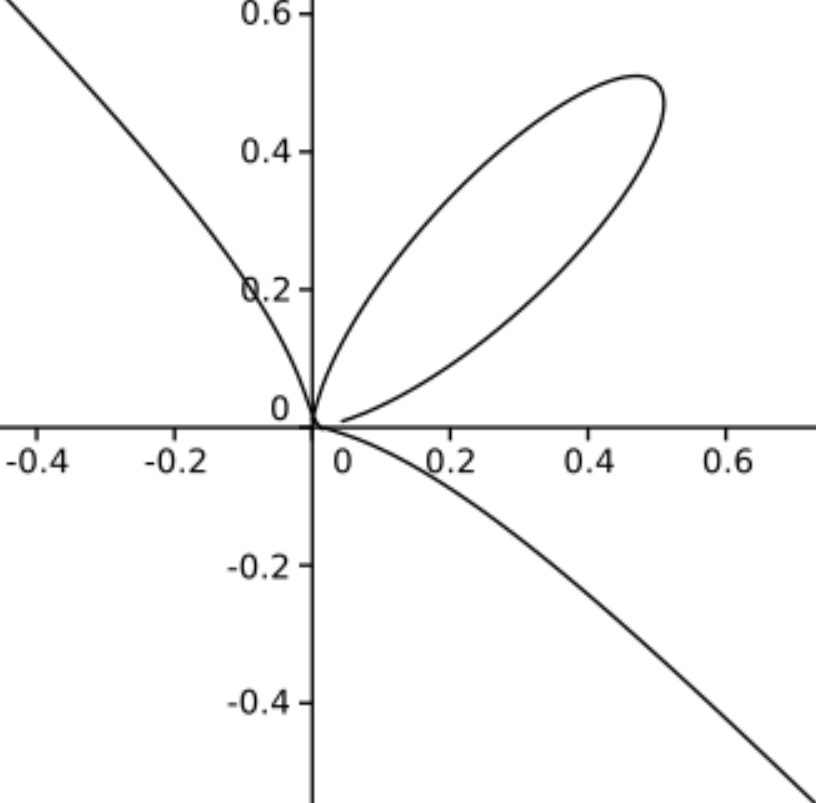}
\caption{$x^5+y^5-x^2y^2=0$}
\label{8oct2020}
\end{center}
\end{figure}
For our computations we have used the following code.
{\tiny
\begin{verbatim}
  LIB "foliation.lib";  
  ring r=0, (x,y),dp; poly f=x5+y5-x2*y2; 
  matrix A=mulmat(f,f); list ll=jordan(A); print(jordanmatrix(ll)); 
  vector v1=[-(-25xy3+5x3+4y2), -25x2y2+6xy]; vector v2=[-(-25x2y2+6xy),-25x3y+5y3+4x2]; vector v3=[-(-5x3y+2y3),-5x4+3xy2];
  vector df=[diff(f,x),diff(f,y)]; vector fdy=[0,f]; vector fdx=[f,0]; module m=v1,v2,v3; 
  division(df,m); division(fdx,m); division(fdy,m); 
  poly disc=Discrim(f); list l=factorize(disc); disc=x^2*(3125x+16); 
  ideal I=std(jacob(f));  poly p=reduce(subst(disc/var(1), var(1),f), I); 
  list l=division(p*f, jacob(f)); l; vector om=[-l[1][2,1], l[1][1,1]]; division(om,m);
  module m2=df,fdx,fdy,om; division(v1, m2);           |      
\end{verbatim}
}
\end{exam}

\def\an{{\rm an}}
\section{A consequence of Quillen-Suslin theorem}
\label{5/11/2020-4}
\def\Sh{{\cal S}}
In Theorem \ref{main1} we have not assumed that $\k\subset\C$ is an algebraically closed field. It turns out that if we use the algebraic closure of $\k$ then $E_f$ is free of rank $2$.
\begin{prop}
\label{27/10/2020}
Let $f\in\k[x,y]$ be a tame polynomial  and assume that $f=0$ is smooth or at most it has  quasi-homogeneous singularities.
 The $\bar\k[x,y]$-module $E_f\otimes_\k\bar\k$ is free of rank $2$. 
\end{prop}
\begin{proof}
The $\k[x,y]$-module $E_f$ is not necessarily projective, however, $E_f\otimes_\k \bar\k $ turns out to be projective. 
The reason is as follows. We compactify $\A^2_\k$ in $\P^2_\k$ and consider the curve $C$ induced by $f=0$ in  $\P^2_\k$. For $\alpha_1=\alpha_2=1$ this curve intersects the line at infinity transversely, and hence, it has no singularity at infinity. However, for arbitrary $\alpha_i$'s it might have singularities at infinity. We perform a desingularization of singularities of $C$ at infinity and get a surface $M$ defined over $\k$ and with a chart $\A^2_\k\subset M$. The curve given by $f=0$ in $\A^2_\k$ induces a curve (we call it again $C$) in $M$ and it has now only smooth points in the compactification divisor $M-\A^2_\k$. Note that for all these we do not need to assume that $\k$ is algebraically closed.
Let $M^{\an}$ be the underlying complex manifold of $M/\k$ for a fixed embedding $\k\subset \C$. 
 
Let $\Sh^{\rm alg}$ be the subsheaf of $\Omega^1_{M/\k}$ containing differential $1$-forms tangent to the curve given by $C$ in $M$.
In a similar way we define $\Sh^{\rm an}\subset \Omega^1_{M^\an}$ in the 
holomorphic context. By definition $\Sh^{\rm an}$ is the analytification of the algebraic sheaf  $\Sh^{\rm alg}\otimes_\k\C$.
Theorem  \ref{2nov2020} applied to quasi-homogeneous singularities implies that $\Sh^{\rm an}$ is a locally free sheaf. Note that here we use the desingularization process as above: the points at infinity of $f=0$ are smooth. 
By Serre's GAGA,  $\Sh^{\rm alg}\otimes_\k\C$ is also a locally free sheaf. Now we look at $\Sh^{\rm alg}\otimes_\k\C$ in the affine chart $\C^2\subset M^{\an}$ and conclude that 
$\Sh^{\rm alg}\otimes_{\k}\bar\k$ is a locally free sheaf in $\A^2_{\bar\k}$. 
For an algebraically closed field $\bar\k$, locally free sheaves over $\A^2_{\bar\k}$ are in one to one correspondence with projective modules: the correspondence is given by taking global sections. Now by Quillen-Suslin theorem, see for instance \cite[Theorem 3.7, page 850]{Langalgebra} projective $\bar\k[x,y]$-modules are free. Note that for Quillen-Suslin theorem we do not need that the base field is algebraically closed, however, for the correspondence with locally free sheaves we need this fact.   

\end{proof}
\begin{rem}\rm
 In general, we do not know how to write the local generators of $\Sh^{\rm alg}$ around each point, however, for smooth curves this is as follows: In the chart $U_0: f_x\not=0$ (resp. $U_1: f_y\not=0$) it is freely generated by $df, fdy$ (resp. $df, fdx$). For instance, $fdx=\frac{f}{f_x}df-\frac{f_y}{f_x}fdy$. If $f=0$ is smooth, $U_0,U_1$ cover $\A^2_\k$. 
\end{rem}

\begin{rem}\rm
 If the module $E_f$ has two elements $\omega_0,\omega_\infty$ with $\omega_0\wedge \omega_\infty=fdx\wedge dy$ then they generate $E_f$ as $\k[x,y]$-module. This is as follows. By Theorem \ref{2nov2020}, the sheaf $\Sh^{\an}$ is locally free, and hence a similar argument as in Proposition 
 \ref{27/10/2020} implies that $E_f\otimes_\k\C$ is generated by two elements $\check\omega_0,\check\omega_\infty$. Moreover, these two elements generate the stalk of $\Sh^\an$ at any point of $\C^2$. It follows that $\check\omega_0\wedge \check\omega_\infty=fdx\wedge dy$. We write $\omega_0, \omega_\infty$ in terms of 
 $\check\omega_0, \check\omega_\infty$ and we get a $2\times 2$ matrix which has non-zero constant determinant. 
\end{rem}

Let $\omega_0,\omega_\infty$ be two generators of the $\bar\k[x,y]$-module $E_f\otimes_\k\bar\k$. It turns out that we have a finite extension $\K$ of $\k$ such that $\omega_0$ and $\omega_\infty$ are defined over $\K$ and $E_f\otimes_\k\K$ is freely generated by $\omega_0,\omega_\infty$.  It is a natural question to bound $[\K:\k]$ in terms of some arithmetic invariants of the curve $C$. 
The case of a circle might be enlightening, see Example \ref{circle2020}.

\section{Examples}
\label{lopes2020}
In this section we explain many examples  of curves $f=0$ such that $E_f$ is generated by two elements (including those in Table \ref{30june2020}). In all these examples $E_f$ is generated by two foliations $\omega_0,\omega_\infty$
with $\omega_0\wedge \omega_\infty=fdx\wedge dy$. The polynomial $1$-form $\omega_f$ is usually big and so we do not reproduce it here. All the figures in Table \ref{30june2020} are the real locus of $f=0$, except for Lins Neto's example, whose figure is an artistic way to depict the arrangement of lines and it is taken from the original article \cite{LinsNeto2002}.  

\begin{exam}\rm
\label{circle2020}
 The module of $1$-forms tangent to $f:=xy-1=0$ is freely generated by 
 $$
 \omega_0:=ydf-fdy=y^2dx+dy,\ \omega_\infty:=fdx,
 $$ 
 which satisfy $\omega_0\wedge \omega_\infty=-fdx\wedge dy$.
 This follows from Proposition \ref{27oct2020} and the identities:
 \begin{equation}
  df=x\omega_0+y\omega_\infty,\ \ \ fdy=f\omega_0-y^2\omega_\infty.
 \end{equation}
 Over $\Q(i)$ this curve  is isomorphic to $f:=x^2+y^2-1=0$.  In fact, the transformation $(x,y)\mapsto (x+iy,x-iy)$ sends the circle $x^2+y^2-1=0$ to $xy-1=0$. 
 The pull-back of the above differential forms is 
\begin{eqnarray*}
\omega_0 &=&\Re(\omega_0)+i\Im(\omega_0):=(x^2-y^2+1)dx+2xydy+i((x^2-y^2-1)dy-2xydx),\\ 
\omega_\infty &=& \Re(\omega_\infty)+i\Im(\omega_\infty)=fdx+ifdy.
\end{eqnarray*}
It would be interesting to prove that the submodule of $\Omega^1_{\A^2_\Q}$ tangent to $x^2+y^2-1$ is not generated by two elements 
defined over $\Q$. %This follows from Proposition \ref{marasindica2020}. 
 For our computations in this example we have used 
 {\tiny
\begin{verbatim}
 LIB "foliation.lib"; 
 ring r=(0,z), (x,y), dp;  minpoly=z^2+1; 
 poly f=x*y-1; matrix P=transpose(MinFol(f,1)[1]); P;
 vector df=[diff(f,x),diff(f,y)];  MinFol(f,1, df); vector fdy=[0,f];  MinFol(f,1, fdy); 
 poly g=x^2+y^2-1;  vector v1=[(x^2-y^2+1),2*x*y]; vector w1=[-2*x*y, (x^2-y^2-1)]; vector v2=[g,0]; vector w2=[0,g];  
 module m=v1+z*w1,v2+z*w2; division(v1,m); division(w1,m); division(v2,m); division(w2,m); 
 division(v1-z*w1,m); division(v2-z*w2,m);  vector dg=[diff(g,x),diff(g,y)];  division(dg,m);  MinFol(g,1, dg);
\end{verbatim}
}
\end{exam}

\begin{exam}[Riccati]\rm
Let us consider the case in which $f$ does not depend on $y$, and hence, $f=0$ is a union of $d$ parallel to 
$y$ axis  lines. In this case $\omega_0=dx$ and $\omega_\infty=f(x)dy$. The Riccati foliations given by  $(p_2(x)y^2+p_1(x)y+p_0(0))dx+f(x)dy$ are in $E_f$. 
\end{exam}

\begin{exam}[Quasi-homogeneous singularities]\rm
These are curves $C:f=0$ given by a homogeneous tame polynomial $f$ of weighted degree $d$ in the weighted ring $\k[x,y],\ \deg(x)=\alpha_1,\ \ deg(y)=\alpha_2$, see Definition \ref{10/11/2020}. This is also the algebraic version of Definition \ref{10nov2020}.
In this case $E_f$ is generated by $\omega_0=df ,\omega_\infty:=\frac{\alpha_1}{d}xdy-\frac{\alpha_2}{d}ydx$.
\end{exam}

\begin{exam}\rm
The arrangement of lines given by
$$
f=(x^a-1)(y^a-1)(x^a-y^a).
$$
for $a=3$ has been studied by Lins Neto in \cite{LinsNeto2002}. In this case $E_f$ is generated by 
$$ 
\omega_0:= -(y^a-1)x^{a-1})dx+(x^a-1)y^{a-1}dy,\ \ \omega_\infty:= -(y^{a+1}-y)dx +(x^{a+1}-x)dy,\ 
$$
The birational $\A_\k ^2 \dashrightarrow \A_\k^2$ given by $(x,y)\mapsto (\frac{1}{x},\frac{1}{y})$ in the affine 
chart $(x,y)$ maps $\omega_0$ to $\omega_\infty$ and vice versa.  For $a=3$ one can prove that $\F(\omega_0+t\omega_\infty),\ \ t\in\P^1_\k$ has a first integral 
if and only if $t$ is a constant in $\Q(e^{\frac{2\pi i}{3}})$. The degree of such a first integral has been computed in \cite{Puchuri2013}.
Another example due to Lins Neto is 
$f=4y^2(1-3x)-4x^3+(3x^2+y^2)^2$ which is up to a linear transformation is the deltoid in Table \ref{30june2020}. 
\end{exam}

\begin{exam}\rm
The arrangement of lines 
given by 
$$f= (x^2-1)(y^2-1)(x^2-(2\tau+1)^2)(y^2-(2\tau +1)^2)\cdot (x^2-y^2)
$$
$$
\cdot(y+1+\frac{1}{\tau} (x+1))(y+1+\tau (x+1))\cdot (y-1+ \tau(x-1))(y-1+\frac{1}{\tau}(x-1))
$$
with $\tau^2-\tau-1=0$ is studied in \cite{peme} and in this case $E_f$ is generated by $\omega_0:=-(y^2-(2\tau+1)^2)(y^2-1)(y+(2\tau-1)x)dx+(x^2-(2\tau+1)^2)(x^2-1)(x+(2\tau-1)y)dy$
and $\omega_\infty$ whose expression can be found in the mentioned reference. Another example from this reference is $f=-1728x^5+720x^3y-80xy^2+64(5x^2-y)^2 + y^3$.  In \cite{KobayashiNaruki} the authors describe the 1-forms $\omega_0,\omega_\infty$ in another coordinate system, and it turn out that 
these are defined over $\Q$.  
\end{exam}

\begin{exam}\rm [Graph of a function]
In this example we consider a curve of the form $y=f(x)$ which is smooth and of genus zero. Let $d:=\deg(f)$. 
For $d=2,3,\cdots,11$ we can verify the following conjecture: For a generic polynomial $f(x)\in\k[x]$ of degree $d$ 
(in a Zariski open subset of the parameter space of $f$), the $\k[x,y]$-module $E_f$ is generated by 
$$
\omega_0:=(ya_1(x)+a_2(x))dx+ a_{3}(x)dy,\ \ \omega_\infty:=(yb_1(x)+b_2(x))dx+b_{3}(x)dy,
$$
where for $d$ even $a_1,a_2,a_3,b_1,b_2,b_3\in \k[x]$ are respectively of degree $\frac{d}{2}-2, \frac{d}{2},  \frac{d}{2}-1,  \frac{d}{2}-1, \frac{d}{2},
\frac{d}{2}$, and for $d$ odd they are respectively of degree  $\frac{d-3}{2}, \frac{d-1}{2},  \frac{d-1}{2},  \frac{d-5}{2}, \frac{d+1}{2},
\frac{d-3}{2}$. This is equivalent to the existence of $a_i$ and $b_i$'s with
$$
a_1b_2-a_2b_1=-f',\ \ \ fa_1+a_2+ f'a_3=0,\ \ fb_1+b_2+ f'b_3=0. 
$$
From these three equalities we can derive
$$
a_1b_3-a_3b_1=1,\ \ a_2b_3-a_3b_2=-f. 
$$
We can also verify the existence of $\omega_0$ and $\omega_\infty$ 
for random choices of $f\in\Q[x,y]$ of higher degree. The three canonical foliations can be written in terms of generators: 
\begin{eqnarray*}
 (y-f)dx &=&b_3\omega_0-a_3\omega_\infty,\\
 (y-f)dy &=&-(yb_1+b_2)\omega_0+(ya_1+a_2)\omega_\infty \\
 d(y-f) &=&-b_1\omega_0+a_1\omega_\infty.
\end{eqnarray*}
For the computations in this example we have used the following code: 
{\tiny
\begin{verbatim}
LIB "foliation.lib";   int d=8;
ring r=(0,t(0..d-2)), (x,y), dp; int i=1; int j;  
poly f=-y+x^d; for (i=0; i<=d-2; i=i+1){f=f+t(i)*x^i;} 
//--Use the next command for a random choice of f. 
//--poly f=RandomPoly(list(x,y),d,-19,10); f=subst(f,y,1); f=f-y;   
matrix P=transpose(MinFol(f,1)[1]); poly co=det(P)/f; co*f-det(P);    
matrix Q1=diff(P,y); intmat degQ1[2][2];    matrix Q2=P-y*Q1; intmat degQ2[2][2];
for (i=1; i<=2; i=i+1)
    {for (j=1; j<=2; j=j+1){degQ1[i,j]=deg(Q1[i,j]); degQ2[i,j]=deg(Q2[i,j]);}}
f; P; degQ1; degQ2; 
\end{verbatim}
}
%vector o1=[f,0]; vector o2=[0,f]; vector o3=[diff(f,var(1)),diff(f,var(2))];
%MinFol(f,1,o3); module M=o1,o2,o3; 
%vector om0=[P[1,1], P[1,2]]; 
%division(om0, M); vector om1=[P[2,1], P[2,2]];  division(om1, M); 
%list vf=P[2,2], -P[2,1]; 
%BadPrV(vf, 40); 
\end{exam}

\begin{exam}[Rose with $k=\frac{1}{2}$]\rm
In this case the foliation $\F(\omega_0)$ has the  first integral $F:=\frac{(x^2 + y^2 - 1/3)^3}{36x^2 + 9y^2 - 4}$.
Its generic fiber is smooth and has two singular points at infinity. 
It has three critical values $t=0, \frac{1}{108}, \infty$. 
The fiber over $t=\frac{1}{108}$ is our initial curve $f=0$ which has single non-degenerated singularity (Milnor number equal to one). 
Since $f=0$ is a rational curve and all the fibers intersects the line at infinity in the same way, we conclude that the genus of a generic fiber of $F$ is one. If we set
$$
 \check \omega_0=(8x-y)dx
+(11x+2y-1)dy,\ \ \check \omega_\infty=(5xy-4y^2+2y)dx+(x^2+10xy-x)dy,\ \
$$
and $\pi:\C^2\to \C^2: (x,y)\mapsto (x^2,y^2)$ then we have
$$
\pi^*\check \omega_0=2\omega_0,\ \ \pi^*\check \omega_\infty=2xy\omega_\infty,
$$
This shows that this example is birational to the case of a quasi-homogeneous singularity. 
We can use Katz-Grothendieck conjecture for vector fields in order to investigate 
whether a foliation by curves has a first integral or not. For this we have written the code {\tt BadPrV} which computes the bad and good primes of a vector field. For instance, before computing the first integral of $\F(\omega_0)$ by hand, we used this in order to be sure that such a foliation has a first integral. 

{\tiny
\begin{verbatim}
LIB "foliation.lib"; 
ring r=(0,t),  (x,y), dp;  
poly f=4*x^4+8*x^2*y^2+4*y^4-4*x^6-12*x^4*y^2-12*x^2*y^4-4*y^6-y^2; 
matrix P=transpose(MinFol(f,1)[1]); P;
poly l=(x^2+ y^2-1/3)^3-(36*x^2 + 9*y^2-4)*t; (P[1,1]*diff(l,y)-P[1,2]*diff(l,x))/l;
list vf=P[1,2], -P[1,1];  BadPrV(vf, 40); 
\end{verbatim}
}
%{\tiny
%\begin{verbatim}
%LIB "foliation.lib"; 
%ring r=(0,t), (x,y), dp;
%poly f=4*x^4 + 8*x^2*y^2 + 4*y^4- 4*x^6 - 12*x^4*y^2 - 12*x^2*y^4 - 4*y^6 - y^2;  //--Rose with k=1/2.
%poly P=(x^2+ y^2-1/3)^3-(36*x^2 + 9*y^2-4)*t;
%number disc=discriminant(P); factorize(substpar(disc,t,x)); substpar(P,t,1/108)/f;
%\end{verbatim}
%}
\end{exam}

The 19th century has produced a lot of curves which are named after many engineers, astronomer and mathematicians. The website  {\tt mathcurve.com} contains a rather full list of such curves. Among all these curves $f=0$, those with $E_f$ generated by two elements, seem to be rare. 
%We examined many curves in this website in order to find more curves such that $E_f$ is generated by two elements. 
The Lissajous and deltoid are among them, as shown  in Table \ref{30june2020}. 
%the curves that we found and the generators of $E_f$ in these two cases are gathered in Table \ref{30june2020}.  

\appendix

\renewcommand{\P}{{\mathbb P}}
\renewcommand{\H}{{\mathbb H}}
\newcommand{\T}{{\mathbb T}}
\newcommand{\ddd}{\textup{d}}
\newcommand{\EE}{{\mathcal E}}
\newcommand{\OO}{{\mathcal O}}
\newcommand{\XX}{{\mathcal X}}
\newcommand{\SSS}{{\mathcal S}}
\newcommand{\www}{\widetilde}
\newcommand{\paa}{\partial}
\newcommand{\Gr}{\textup{Gr}}
\newcommand{\id}{\textup{id}}
\newcommand{\imm}{\textup{im\,}}

\section{Nonexistence of a curve singularity $f$ where all Jordan blocks of the multiplication by $f$ in its Milnor
vector space have size $2$ (By Claus Hertling)}
\label{24/11/2020}
\bigskip
\begin{prop}
\label{27.11.2020}
 Let $f\in\C\{x,y\}$ be a holomorphic function germ with
$f(0)=0$ and with an isolated singularity at $0\in\C^2$. 
Let $A_f$ be the multiplication by $f$ in its
Milnor vector space $Q_f:=\C\{x,y\}/\textup{jacob}(f)$.
Then $A_f$ is nilpotent and all Jordan blocks of $A_f$
have size 1 or 2. Not all Jordan blocks of $A_f$ have size 2.
\end{prop}
Let us call a singularity $f\in\C\{x,y\}$ such that
all Jordan blocks of $A_f$ have size 2 
{\it Jordan block extreme}. Proposition \ref{27.11.2020} says that
such singularities do not exist. 
If they would exist, one could allow in Theorem \ref{main1} that either
all singularities on $C$ are quasi-homogeneous or that
all singularities on $C$ are Jordan block extreme. 
This follows from Theorem \ref{5feb07}. 

The nonexistence of Jordan block extreme singularities 
is not really surprising. But it is also not trivial.
The following proof uses the Gauss-Manin connection,
the spectral numbers and the fact that Hertling's 
variance conjecture on the spectral numbers is true 
in the case of curve singularities \cite{Br04}.
The idea is to show that a Jordan block extreme curve 
singularity would have spectral numbers which violate 
the variance conjecture. 

\bigskip
{\bf Proof of Proposition \ref{27.11.2020}:}
First we have to review the spectral numbers and their 
background. For this, we start with an arbitrary 
holomorphic function germ 
$f\in\OO_{\C^{n+1},0}=\C\{x_0,...,x_n\}$ 
with $f(0)=0$ and with an isolated singularity at 0,
here $n\geq 1$. Two references in book form for the
following are \cite{arn} and \cite{her}. 

Choose a good representative $\tilde{f}:X\to\Delta$,
with $B_\varepsilon\subset\C^{n+1}$ a small ball around 0,
$\Delta=\Delta_\delta\subset \C$ a very small disk around 0 
(so $0<\delta\ll\varepsilon\ll 1$) and
$X:=B_{\varepsilon}\cap f^{-1}(\Delta)$.
Consider its {\it cohomology bundle}
$\bigcup_{t\in\Delta^*}H^n(\tilde{f}^{-1}(t),\C)$, 
here $\Delta^*:=\Delta-\{0\}$. It is a flat vector bundle
of rank $\mu\in\N$.
The covariant derivative on it with respect to the coordinate
vector field of the coordinate $t$ on $\Delta$ 
is called $\paa_t$. We define
\begin{eqnarray*}
Q_f&:=& \OO_{\C^{n+1},0}/(\textup{jacob}(f))
\quad\textup{Milnor vector space},\\
\Omega_f&:=& \Omega^{n+1}_{\C^{n+1},0}/
df\land \Omega^n_{\C^{n+1},0},\\
H_0''&:=& \Omega^{n+1}_{\C^{n+1},0}/
df\land d\Omega^{n-1}_{\C^{n+1},0}
\quad\textup{Brieskorn lattice},\\
H_0'&:=& df\land \Omega^n_{\C^{n+1},0}/
df\land d\Omega^{n-1}_{\C^{n+1},0},
\end{eqnarray*}
\begin{eqnarray*}
V^{>-\infty}&:=&\{\textup{germs at 0 of sections in the
cohomology bundle of moderate growth}\},\\
C^\alpha&:=&\{\sigma\in V^{>-\infty}\,|\, 
(t\partial_t-\alpha)^{n+1}(\sigma)=0\},\\
V^\alpha&:=& \bigoplus_{\beta\in [\alpha,\alpha+1)}
\C\{t\}\cdot C^\beta,\\
V^{>\alpha}&:=& \bigoplus_{\beta\in (\alpha,\alpha+1]}
\C\{t\}\cdot C^\beta.
\end{eqnarray*}
$Q_f$ and $\Omega_f$ are $\C$-vector spaces of dimension $\mu$.
$V^{>-\infty}$ is $\C\{t\}[t^{-1}]$-vector space of 
dimension $\mu$. It contains the spaces $C^\alpha$,
$V^\alpha$, $V^{>\alpha}$, $H_0''$ and $H_0'$.  
For any $\alpha\in\R$
$$V^{>-\infty}=\bigoplus_{\beta\in[\alpha,\alpha+1)}
\C\{t\}[t^{-1}]C^\beta,$$
$V^\alpha$, $V^{>\alpha}$, $H_0''$ and
$H_0'$ are free $\C\{t\}$-modules of rank $\mu$.
And $V^\alpha$ for $\alpha>-1$, $V^{>\alpha}$ for 
$\alpha\geq-1$, $H_0''$ and $H_0'$ are free
$\C\{\{\paa_t^{-1}\}\}$-modules of rank $\mu$
($\C\{\{\paa_t^{-1}\}\}$ is the ring of power series of
Gevrey class 1 in $\paa_t^{-1}$.) The Brieskorn lattice
$H_0''$ satisfies 
$$V^{>-1}\supset H_0''\supset V^{n-1}.$$
We have the $\C$-vector space isomorphisms
\begin{eqnarray*}
t:V^\alpha&\to& V^{\alpha+1},\\
t:V^{>\alpha}&\to& V^{>\alpha+1},\\
\paa_t^{-1}:V^\alpha&\to& V^{\alpha+1}\quad\textup{for }
\alpha>-1,\\
\paa_t^{-1}:V^{>\alpha}&\to& V^{>\alpha+1}\quad\textup{for }
\alpha\geq -1,\\
\paa_t^{-1}:H_0''&\to& H_0',\\
\Omega_f&\cong& H_0''/H_0',\\
C^\alpha&\cong& V^\alpha/V^{>\alpha}=:
\textup{Gr}_V^\alpha V^{>-\infty}.
\end{eqnarray*}

The space $\Omega_f$ inherits from $H_0''\subset V^{>-1}$
a $V$-filtration, as follows, 
\begin{eqnarray*}
V^\alpha\Omega_f:= \frac{V^\alpha\cap H_0''+H_0'}{H_0'}
\subset \frac{H_0''}{H_0'}=\Omega_f,
\end{eqnarray*}
with quotients
\begin{eqnarray*}
\Gr_V^\alpha\Omega_f = 
\frac{V^\alpha\Omega_f}{V^{>\alpha}\Omega_f}
\cong \frac{\Gr_V^\alpha H_0''}{\Gr_V^\alpha H_0'}
\cong\frac{H_0''\cap V^\alpha+V^{>\alpha}}
{H_0'\cap V^\alpha + V^{>\alpha}}.
\end{eqnarray*}
The spectral numbers of the singularity $f$ 
are the $\mu$ rational numbers
$\alpha_1,...,\alpha_\mu$ with 
\begin{eqnarray*}
\alpha_1\leq ...\leq\alpha_\mu\quad \textup{and}\quad
\dim \Gr_V^\alpha\Omega_f =\sharp\{j\in\{1,....,\mu\}\,|\,
\alpha_j=\alpha\}.
\end{eqnarray*}
The inclusions $V^{>-1}\supset H_0''\supset V^{n-1}$
and the isomorphism $\Omega_f=H_0''/\paa_t^{-1}H_0''$
show
$$-1<\alpha_1\leq ...\leq \alpha_\mu <n.$$
Also the symmetry 
$$\alpha_j+\alpha_{\mu+1-j}=n-1, \quad 
\textup{or, equivalenty,}\quad \dim\Gr_V^\alpha\Omega_f
=\dim\Gr_V^{n-1-\alpha}\Omega_f$$
holds (one proof of it uses Varchenko's description via
$H_0''$ of Steenbrink's mixed Hodge structure,
another one uses K. Saito's higher residue pairings).
The following {\it variance inequality} for the spectral
numbers was conjectured by Hertling \cite{her},
$$\mu^{-1}\sum_{j=1}^\mu (\alpha_j-\frac{n-1}{2})^2
\leq \frac{\alpha_\mu-\alpha_1}{12}.$$
For curve singularities (i.e. $n=1$), it was proved by 
Br\'elivet \cite{Br04}. 

The action of $t$ on $V^{>-\infty}$ induces an action
on $\Omega_f$ which is called $t_{\Omega_f}$. 
It satisfies
$$t_{\Omega_f}:V^\alpha\Omega_f\to V^{\alpha+1}\Omega_f.$$
Therefore $t_{\Omega_f}$ is nilpotent. Because of 
$V^{\alpha_1}\Omega_f=\Omega_f$ and $V^{>\alpha_\mu}\Omega_f=0$
and $\alpha_1+(n+1)>\alpha_\mu$, $t_{\Omega_f}$
has at most Jordan blocks of size $n+1$. 

The $\mu$-dimensional $\C$-vector spaces $Q_f$ and
$\Omega_f$ are not canonically isomorphic.
But the choice of any volume form 
$u(x){\rm d}x_0...{\rm d}x_n=u(x){\rm d}x$ 
(volume form: $u(0)\neq 0$, i.e. $u(x)\in\OO^*_{\C^{n+1},0}$)
induces an isomorphism
$$Q_f\to\Omega_f,\quad g\mapsto g\cdot [u(x){\rm d}x],$$
which commutes with the multiplication by $f$ respectively $t$,
\begin{eqnarray*}
\begin{matrix}
Q_f & \stackrel{\cong}{\to} & \Omega_f & 
g\mapsto g\cdot [u(x){\rm d}x]\\
 \downarrow A_f & & \downarrow t_{\Omega_f} & \\
Q_f & \stackrel{\cong}{\to} & \Omega_f & 
g\mapsto g\cdot [u(x){\rm d}x]
\end{matrix}
\end{eqnarray*}
Therefore $A_f$ and $t_{\Omega_f}$ have the same Jordan block
structure. Now the review of the spectral numbers and their
background is finished. 

Now we suppose $n=1$, so $f$ is a curve singularity,
and we suppose that $f$ is Jordan block extreme,
i.e. all Jordan blocks of $A_f$ and $t_{\Omega_f}$
have size 2. We will come to a contradiction.
$\mu$ is even. We have
$$V^{>\alpha_\mu-1}\Omega_f\subset \ker t_{\Omega_f} 
=\imm t_{\Omega_f} = t_{\Omega_f}(\Omega_f)
=t_{\Omega_f}(V^{\alpha_1}\Omega_f)
\subset V^{\alpha_1+1}\Omega_f.$$
Together with $\alpha_1+1>0>\alpha_\mu-1$, so 
$V^{>\alpha_\mu-1}\Omega_f\supset 
V^0\Omega_f\supset V^{\alpha_1+1}\Omega_f$,
this shows
$$V^{>\alpha_\mu-1}\Omega_f= \ker t_{\Omega_f} 
= \imm t_{\Omega_f}= V^{\alpha_1+1}\Omega_f$$
and
$$\{\alpha_1,...,\alpha_{\mu/2}\}\subset 
[\alpha_1,\alpha_\mu-1],\quad
\{\alpha_{\mu/2+1},...,\alpha_\mu\}\subset
[\alpha_1+1,...,\alpha_\mu].$$
This distribution of the spectral numbers is already strange.
But for a contradiction to the variance inequality, we
need more. 

For any element $v\in\Omega_f-\{0\}$ denote by 
$\gamma(v)\in\{\alpha_1,...,\alpha_\mu\}$ the unique
number $\gamma$ such that 
$v\in V^\gamma\Omega_f-V^{>\gamma}\Omega_f$,
and then denote by 
$\Gr_V^{\gamma(v)}v\in \Gr_V^{\gamma(v)}\Omega_f-\{0\}$
the class of $v$ in $\Gr_V^{\gamma(v)}\Omega_f$.

We choose elements $v_1,...,v_{\mu/2}\in\Omega_f$
such that $\gamma(v_j)=\alpha_j$ and such that the classes 
$\Gr_V^{\alpha_1}v_1,...,\Gr_V^{\alpha_{\mu/2}}v_{\mu/2}$
form a $\C$-vector space basis of the sum of quotients
$\sum_{\alpha\in\{\alpha_1,...,\alpha_{\mu/2}\}}
Gr_V^\alpha\Omega_f$. 
Then any linear combination $v=\sum_{j=1}^{\mu/2}\lambda_jv_j$
with $(\lambda_1,...,\lambda_{\mu/2}) \neq 0$ is nonzero and has 
$\gamma(v)\in\{\alpha_1,...,\alpha_{\mu/2}\}$, so it is 
not in $V^{>\alpha_\mu-1}=\ker t_{\Omega_f}$. 
Therefore the vector space
$\sum_{j=1}^{\mu/2}\C\cdot v_j$ has dimension $\mu/2$, and
$$\Omega_f = \Bigl(\bigoplus_{j=1}^{\mu/2}\C\cdot v_j\Bigr)
\oplus \ker t_{\Omega_f}.$$
Recall $\ker t_{\Omega_f}=\imm t_{\Omega_f}$ and that this
subspace of $\Omega_f$ has dimension $\mu/2$. Therefore
$$\ker t_{\Omega_f}=\imm t_{\Omega_f}=
\bigoplus_{j=1}^{\mu/2}\C\cdot t_{\Omega_f}(v_j).$$
Obviously $\gamma(t_{\Omega_f}(v_j))\geq \alpha_j+1$,
but equality does not necessarily hold, and the classes 
$\Gr_V^{\gamma(t_{\Omega_f}(v_j))}t_{\Omega_f}(v_j)$ for
$j\in\{1,...,\mu/2\}$ are not necessarily linearly independent.

The following claim replaces the basis 
$t_{\Omega_f}(v_1),...,t_{\Omega_f}(v_{\mu/2})$ of 
$\imm t_{\Omega_f}$ by a basis $w_1,...,w_{\mu/2}$ of
$\imm t_{\Omega_f}$ which fits better to the spectral numbers
$\alpha_{\mu/2+1},...\alpha_\mu$. 

\medskip
{\bf Claim:} (a) There is a lower triangular matrix
$(a_{ij})\in M_{\mu/2\times \mu/2}(\C)$ with $a_{ii}=1$
such that the basis 
$$(w_1,...,w_{\mu/2})=(t_{\Omega_f}(v_1),...,
t_{\Omega_f}(v_{\mu/2}))\cdot (a_{ij})$$ 
of $\imm t_{\Omega_f}$ satisfies the
following: Write $\beta_j:=\gamma(w_j)$. The classes 
$\Gr_V^{\beta_1}w_1,...,\Gr_V^{\beta_{\mu/2}}w_{\mu/2}$
are linearly independent.

(b) Therefore they form a basis of 
$\bigoplus_{\alpha\in\{\alpha_{\mu/2+1},...,\alpha_\mu\}}
\Gr_V^\alpha\Omega_f$, and therefore there is a bijection
$\sigma:\{1,...,\mu/2\}\to\{\mu/2+1,...,\mu\}$
with $\alpha_{\sigma(j)}=\beta_j$ for $j\in\{1,...,\mu/2\}$.

(c) $\beta_j\geq \alpha_j+1$ for $j\in\{1,...,\mu/2\}$. 

\medskip
{\bf Proof of the Claim:}
(a) The vectors $w_j$ are constructed inductively in the
order $w_{\mu/2},w_{\mu/2-1},...,w_1$. The first step
$w_{\mu/2}=t_{\Omega_f}(v_{\mu/2})$ is trivial.
Suppose that the vectors $w_{j+1},...,w_{\mu/2}$
(and the corresponding entries $a_{im}$) have been constructed.
One constructs $w_j$ be a sequence of steps which give
$w_j^{(0)}:=t_{\Omega_f}(v_j),w_j^{(1)},...,w_j^{(k)}=w_j$
for some $k\geq 0$. Each of these vectors is in
$t_{\Omega_f}(v_j)+\bigoplus_{i\geq j+1}\C\cdot w_i$. 
Suppose that $w_j^{(l)}$ has been constructed. If 
$$Gr_V^{\gamma(w_j^{(l)})}w_j^{(l)}\notin 
\bigoplus_{i:\, i\geq j+1,\beta_i=\gamma(w_j^{(l)})}
\C\cdot \Gr_V^{\beta_i}w_i,$$
then $l=k$ and $w_j^{(l)}=w_j$. If 
$$Gr_V^{\gamma(w_j^{(l)})}w_j^{(l)}\in 
\bigoplus_{i:\, i\geq j+1,\beta_i=\gamma(w_j^{(l)})}
\C\cdot \Gr_V^{\beta_i}w_i,$$
one adds to $w_j^{(l)}$ a suitable linear combination
$\sum_{i:\, i\geq j+1,\beta_i=\gamma(w_j^{(l)})}
a_{ij}\cdot w_i$, such that the sum $w_j^{(l+1)}$
satisfies $\gamma(w_j^{(l+1)})>\gamma(w_j^{(l)})$.
This construction stops at some $w_j^{(k)}=w_j$.

(b) This follows immediately from part (a).

(c) The construction in the proof of part (a) gives
$$\beta_j=\gamma(w_j)=\gamma(w_j^{(k)})>..>\gamma(w_j^{(0)})
=\gamma(t_{\Omega_f}(v_j))\geq \alpha_j+1$$
(in the case $k=0$ without the strict inequalities).
This finishes the proof of the Claim.
\hfill ($\Box$)

\medskip
Now we can estimate the variance of the spectral numbers.
It is 
\begin{eqnarray*}
&&\mu^{-1}\sum_{j=1}^\mu\alpha_j^2
=\mu^{-1}\sum_{j=1}^{\mu/2}(\alpha_j^2+\beta_j^2)
=\mu^{-1}\sum_{j=1}^{\mu/2}
\Bigl((\alpha_j+\frac{1}{2}-\frac{1}{2})^2 + 
(\beta_j-\frac{1}{2}+\frac{1}{2})^2\Bigr)\\
&=&\mu^{-1}\sum_{j=1}^{\mu/2}
\Bigl((\alpha_j+\frac{1}{2})^2 + (\alpha_j+\frac{1}{2})(-1)
+\frac{1}{4} +(\beta_j-\frac{1}{2})^2 + 
(\beta_j-\frac{1}{2})\cdot 1 + \frac{1}{4}\Bigr)\\
&=&\mu^{-1}\sum_{j=1}^{\mu/2}
\Bigl(\frac{1}{2}+ (\alpha_j+\frac{1}{2})^2
+(\beta_j-\frac{1}{2})^2
+(\beta_j-\alpha_j-1)\Bigr)\\
&\geq & \mu^{-1}\sum_{j=1}^{\mu/2}\frac{1}{2}
=\frac{1}{4}\qquad 
\textup{(here part (c) of the claim is used).}
\end{eqnarray*}
This does not fit to the variance inequality
for curve singularities \cite{Br04},
$$\mu^{-1}\sum_{j=1}^\mu\alpha_j^2\leq
\frac{\alpha_\mu-\alpha_1}{12}< \frac{2}{12}=\frac{1}{6}.$$
We arrive at a contradiction. A Jordan block extreme
curve singularity does not exist.
\hfill$\Box$

%\section{Computer codes}
%\label{computercodes}
%For the computations in this paper we have written the procedures
%{\tt SyzFol, MinFol, BadPrV} of {\tt foliation.lib} which is available in the second author's webpage.
%\footnote{ {\tt http://w3.impa.br/$\sim$hossein/foliation-allversions/foliation.lib}  }

%\bibliography{biblio}

\begin{thebibliography}{{Med}13}

\bibitem[AGV85]{arn1}
V.~I. {Arnold}, S.~M. {Gusein-Zade}, and A.~N. {Varchenko}.
\newblock {\em {Singularities of differentiable maps, Volume I. Classification
  of critical points, caustics and wave fronts.}}
\newblock Boston, MA: Birkh\"auser, 1985.

\bibitem[AGZV88]{arn}
V.~I. Arnold, S.~M. Gusein-Zade, and A.~N. Varchenko.
\newblock {\em Singularities of differentiable maps. Monodromy and asymptotics
  of integrals {V}ol. {II}}, volume~83 of {\em Monographs in Mathematics}.
\newblock Birkh\"auser Boston Inc., Boston, MA, 1988.

\bibitem[Br{\'e}04]{Br04}
Thomas Br{\'e}livet.
\newblock The {H}ertling conjecture in dimension 2.
\newblock {\em math/0405489}, 2004.

\bibitem[CLO05]{Cox2005}
David~A. {Cox}, John {Little}, and Donal {O'Shea}.
\newblock {\em {Using algebraic geometry. 2nd ed}}, volume 185.
\newblock New York, NY: Springer, 2nd ed. edition, 2005.

\bibitem[CLS92]{CLS92}
C.~{Camacho}, Alcides {Lins Neto}, and P.~{Sad}.
\newblock {Foliations with algebraic limit sets}.
\newblock {\em {Ann. Math. (2)}}, 136(2):429--446, 1992.

\bibitem[GH94]{gri}
Phillip Griffiths and Joseph Harris.
\newblock {\em Principles of algebraic geometry}.
\newblock Wiley Classics Library. John Wiley \& Sons Inc., New York, 1994.
\newblock Reprint of the 1978 original.

\bibitem[GPS01]{GPS01}
G.-M. Greuel, G.~Pfister, and H.~Sch\"onemann.
\newblock {\sc Singular} 2.0.
\newblock {A Computer Algebra System for Polynomial Computations}, Centre for
  Computer Algebra, University of Kaiserslautern, 2001.
\newblock {\tt http://www.singular.uni-kl.de}.

\bibitem[Her02]{her}
Claus Hertling.
\newblock {\em Frobenius manifolds and moduli spaces for singularities}, volume
  151 of {\em Cambridge Tracts in Mathematics}.
\newblock Cambridge University Press, Cambridge, 2002.

\bibitem[KN88]{KobayashiNaruki}
Ryoichi {Kobayashi} and Isao {Naruki}.
\newblock {Holomorphic conformal structures and uniformization of complex
  surfaces}.
\newblock {\em {Math. Ann.}}, 279(3):485--500, 1988.

\bibitem[{Lan}02]{Langalgebra}
Serge {Lang}.
\newblock {\em {Algebra. 3rd revised ed}}, volume 211.
\newblock New York, NY: Springer, 3rd revised ed. edition, 2002.

\bibitem[{Lin}02]{LinsNeto2002}
A.~{Lins Neto}.
\newblock {Some examples for the Poincar\'e and Painlev\'e problems.}
\newblock {\em {Ann. Sci. \'Ec. Norm. Sup\'er. (4)}}, 35(2):231--266, 2002.

\bibitem[{Med}13]{Puchuri2013}
Liliana~Puchuri {Medina}.
\newblock {Degree of the first integral of a pencil in \(\mathbb{P}^2\) defined
  by Lins Neto}.
\newblock {\em {Publ. Mat., Barc.}}, 57(1):123--137, 2013.

\bibitem[Mov07]{mo}
H.~Movasati.
\newblock Mixed {H}odge structure of affine hypersurfaces.
\newblock {\em Ann. Inst. Fourier (Grenoble)}, 57(3):775--801, 2007.

\bibitem[Mov19]{ho13}
H.~Movasati.
\newblock {\em A {C}ourse in {H}odge {T}heory: with {E}mphasis on {M}ultiple
  {I}ntegrals}.
\newblock Available at
  \href{http://w3.impa.br/~hossein/myarticles/hodgetheory.pdf }{author's
  webpage}. 2019.

\bibitem[MP05]{peme}
L.~G. {Mendes} and J.~V. {Pereira}.
\newblock {Hilbert modular foliations on the projective plane}.
\newblock {\em {Comment. Math. Helv.}}, 80(2):243--291, 2005.

\bibitem[{Sai}71]{Saito1971}
Kyoji {Saito}.
\newblock {Quasihomogene isolierte Singularit\"aten von Hyperfl\"achen}.
\newblock {\em {Invent. Math.}}, 14:123--142, 1971.

\bibitem[Sai76]{sa76}
Kyoji Saito.
\newblock On a generalization of de-{R}ham lemma.
\newblock {\em Ann. Inst. Fourier (Grenoble)}, 26(2):vii, 165--170, 1976.

\bibitem[{Sai}80]{KSaito1980}
Kyoji {Saito}.
\newblock {Theory of logarithmic differential forms and logarithmic vector
  fields.}
\newblock {\em {J. Fac. Sci., Univ. Tokyo, Sect. I A}}, 27:265--291, 1980.

\end{thebibliography}
%\bibliographystyle{alpha}

\def\cprime{$'$} \def\cprime{$'$} \def\cprime{$'$} \def\cprime{$'$}

%---------------------------------------------------
\end{document}